\documentclass[11pt,a4paper,twoside,english]{article}
\usepackage[colormode=ntnu]{article-style}
\usepackage[font=normalsize,labelfont={color=artcolor}]{caption}
\usepackage{subcaption}
\RequirePackage{tikz}
\RequirePackage{pgfplots}
\pgfplotsset{compat=newest}
\pgfplotsset{colormap={CM}{rgb=(0,0,1) color=(black) rgb255=(238,140,238)}}
\usetikzlibrary{arrows.meta}
\usepackage{pgfplotstable}
\usepackage{wrapfig}
\newlength\figureheight 
\newlength\figurewidth

\hypersetup{bookmarksopenlevel=2}

\makeatletter
\renewcommand*\env@matrix[1][c]{\hskip -\arraycolsep
  \let\@ifnextchar\new@ifnextchar
  \array{*\c@MaxMatrixCols #1}}
\makeatother

\makeatletter
\def\hdots@for#1#2{\multicolumn{#2}c%
  {\m@th\dotsspace@1.5mu\mkern-#1\dotsspace@
   \xleaders\hbox{$\m@th\mkern#1\dotsspace@\cdot\mkern#1\dotsspace@$}%
           \hskip\z@\@plus 1filll
   \mkern-#1\dotsspace@}%
   }
\makeatother

\RequirePackage{etoolbox}
\makeatletter
\patchcmd{\@adjustvertspacing}
  {\jot\baselineskip \divide\jot 4}
  {\jot=.5\baselineskip}
  {}{}
\makeatother



\newcommand*{\diffdchar}{\textnormal{d}}
\newcommand*{\dee}{\mathop{\diffdchar\!}}

\DeclareMathOperator*{\domain}{dom}
\DeclareMathOperator{\range}{ran}

\DeclareMathOperator{\interior}{int}

\DeclareMathOperator{\support}{supp}

\DeclareMathOperator*{\essinf}{ess\,inf}
\DeclareMathOperator*{\esssup}{ess\,sup}

\DeclareMathOperator*{\argmin}{arg\,min}

\DeclareMathOperator*{\esslim}{ess\,lim}

\DeclareMathOperator{\sign}{sgn}

\DeclarePairedDelimiterX{\norm}[1]{\lVert}{\rVert}{\ifblank{#1}{\,\cdot\,}{#1}}
\DeclarePairedDelimiterX{\seminorm}[1]{\lvert}{\rvert}{\ifblank{#1}{\,\cdot\,}{#1}}
\DeclarePairedDelimiterX{\abs}[1]{\lvert}{\rvert}{\ifblank{#1}{\,\cdot\,}{#1}}
\DeclarePairedDelimiterX{\innerproduct}[2]{\langle}{\rangle}{\ifblank{#1}{\,\cdot\,}{#1},\ifblank{#2}{\,\cdot\,}{#2}}
\DeclarePairedDelimiterX{\dualitypairing}[2]{\langle}{\rangle}{\ifblank{#1}{\,\cdot\,}{#1},\ifblank{#2}{\,\cdot\,}{#2}}
\DeclarePairedDelimiterXPP{\distance}[2]{\operatorname{dist}}{(}{)}{}{\ifblank{#1}{\,\cdot\,}{#1},\ifblank{#2}{\,\cdot\,}{#2}}

\providecommand\given{}
 \newcommand\SetSymbol[1][]{%
      \nonscript\: :
      \allowbreak
      \nonscript\:
      \mathopen{}}
   \DeclarePairedDelimiterX\Set[1]\{\}{%
      \renewcommand\given{\SetSymbol[\delimsize]}
      #1
   }

\DeclareMathOperator*{\identityoperator}{id}

\newcommand{\tothepowerof}{\textnormal{\textasciicircum}}

\makeatletter
\newcommand{\biggg}{\bBigg@{2.6}}
\newcommand{\Biggg}{\bBigg@{2.87}}
\newcommand{\vast}{\bBigg@{3}}
\newcommand{\Vast}{\bBigg@{5}}
\makeatother

\renewcommand\leq\leqslant
\renewcommand\geq\geqslant

\DeclareRobustCommand{\A}[0]{\mathbb{A}}
\DeclareRobustCommand{\B}[0]{\mathbb{B}}
\DeclareRobustCommand{\C}[0]{\mathbb{C}}
\DeclareRobustCommand{\D}[0]{\mathbb{D}}
\DeclareRobustCommand{\E}[0]{\mathbb{E}}
\DeclareRobustCommand{\F}[0]{\mathbb{F}}
\DeclareRobustCommand{\G}[0]{\mathbb{G}}
\DeclareRobustCommand{\H}[0]{\mathbb{H}}
\DeclareRobustCommand{\I}[0]{\mathbb{I}}
\DeclareRobustCommand{\J}[0]{\mathbb{J}}
\DeclareRobustCommand{\K}[0]{\mathbb{K}}
\DeclareRobustCommand{\L}[0]{\mathbb{L}}
\DeclareRobustCommand{\M}[0]{\mathbb{M}}
\DeclareRobustCommand{\N}[0]{\mathbb{N}}
\DeclareRobustCommand{\O}[0]{\mathbb{O}}
\DeclareRobustCommand{\P}[0]{\mathbb{P}}
\DeclareRobustCommand{\Q}[0]{\mathbb{Q}}
\DeclareRobustCommand{\R}[0]{\mathbb{R}}
\DeclareRobustCommand{\S}[0]{\mathbb{S}}
\DeclareRobustCommand{\T}[0]{\mathbb{T}}
\DeclareRobustCommand{\U}[0]{\mathbb{U}}
\DeclareRobustCommand{\V}[0]{\mathbb{V}}
\DeclareRobustCommand{\W}[0]{\mathbb{W}}
\DeclareRobustCommand{\X}[0]{\mathbb{X}}
\DeclareRobustCommand{\Y}[0]{\mathbb{Y}}
\DeclareRobustCommand{\Z}[0]{\mathbb{Z}}

\makeatletter
\let\save@mathaccent\mathaccent
\newcommand*\if@single[3]{%
  \setbox0\hbox{${\mathaccent"0362{#1}}^H$}%
  \setbox2\hbox{${\mathaccent"0362{\kern0pt#1}}^H$}%
  \ifdim\ht0=\ht2 #3\else #2\fi
  }
\newcommand*\rel@kern[1]{\kern#1\dimexpr\macc@kerna}
\newcommand*\widebar[1]{\@ifnextchar^{{\wide@bar{#1}{0}}}{\wide@bar{#1}{1}}}
\newcommand*\wide@bar[2]{\if@single{#1}{\wide@bar@{#1}{#2}{1}}{\wide@bar@{#1}{#2}{2}}}
\newcommand*\wide@bar@[3]{%
  \begingroup
  \def\mathaccent##1##2{%
    \let\mathaccent\save@mathaccent
    \if#32 \let\macc@nucleus\first@char \fi
    \setbox\z@\hbox{$\macc@style{\macc@nucleus}_{}$}%
    \setbox\tw@\hbox{$\macc@style{\macc@nucleus}{}_{}$}%
    \dimen@\wd\tw@
    \advance\dimen@-\wd\z@
    \divide\dimen@ 3
    \@tempdima\wd\tw@
    \advance\@tempdima-\scriptspace
    \divide\@tempdima 10
    \advance\dimen@-\@tempdima
    \ifdim\dimen@>\z@ \dimen@0pt\fi
    \rel@kern{0.6}\kern-\dimen@
    \if#31
      \overline{\rel@kern{-0.6}\kern\dimen@\macc@nucleus\rel@kern{0.4}\kern\dimen@}%
      \advance\dimen@0.4\dimexpr\macc@kerna
      \let\final@kern#2%
      \ifdim\dimen@<\z@ \let\final@kern1\fi
      \if\final@kern1 \kern-\dimen@\fi
    \else
      \overline{\rel@kern{-0.6}\kern\dimen@#1}%
    \fi
  }%
  \macc@depth\@ne
  \let\math@bgroup\@empty \let\math@egroup\macc@set@skewchar
  \mathsurround\z@ \frozen@everymath{\mathgroup\macc@group\relax}%
  \macc@set@skewchar\relax
  \let\mathaccentV\macc@nested@a
  \if#31
    \macc@nested@a\relax111{#1}%
  \else
    \def\gobble@till@marker##1\endmarker{}%
    \futurelet\first@char\gobble@till@marker#1\endmarker
    \ifcat\noexpand\first@char A\else
      \def\first@char{}%
    \fi
    \macc@nested@a\relax111{\first@char}%
  \fi
  \endgroup
}
\makeatother
\begin{document}
\title{Subgradient-based Lavrentiev regularisation\\ of monotone ill-posed problems}
\author{Markus Grasmair\,\orcid{0000-0002-6116-5584}\email{markus.grasmair@ntnu.no} \and Fredrik Hildrum\,\orcid{0000-0002-9905-1670}\email{arxiv@fredrik.hildrum.net}}
\address{Department of Mathematical Sciences,\\ NTNU -- Norwegian University of Science and Technology,\\ 7491 Trondheim, Norway}
\titlerunning{Subgradient-based Lavrentiev regularisation}
\authorrunning{M.~Grasmair and F.~Hildrum}
\date{\today}
\maketitle

\titlegraphic{%
\colorlet{artcolor-titlegraphic}{artcolor!75!artcolorbg}%
\vspace*{-0.8em}%
\centering%
\setlength\figureheight{.1\textwidth}%
\setlength\figurewidth{.4\textwidth}%
\input{tikz/title_volterra.tikz}%
\hspace{.04\textwidth}\raisebox{1em}{\textcolor{artcolor-titlegraphic}{\footnotesize\({ \xrightarrow[]{\textsc{ denoise }} }\)}}\hspace{.04\textwidth}%
%
%
\colorlet{mycolor1}{artcolor-titlegraphic}%
\colorlet{mycolor2}{artcolortext}%
\begin{tikzpicture}

\begin{axis}[%
width=\figurewidth,
height=\figureheight,
at={(0\figurewidth,0\figureheight)},
scale only axis,
hide axis,
xmin=-0.025,
xmax=1.025,
ymin=-17,
ymax=10.1
]

\addplot [color=mycolor1,solid,line width=1.5pt,forget plot]
  table[row sep=crcr]{%
0	-7.66092233612626\\
0.029029029029029	-7.66092233612626\\
0.03003003003003	3.42497544479183\\
0.032032032032032	3.42497544479183\\
0.033033033033033	5.41625836293544\\
0.036036036036036	5.41625836293544\\
0.037037037037037	6.71166106155351\\
0.0820820820820821	6.71166106155351\\
0.0830830830830831	6.39359809173781\\
0.0990990990990991	6.39359809173781\\
0.1001001001001	1.37564853019105\\
0.101101101101101	1.37564853019105\\
0.102102102102102	-4.44796654431455\\
0.119119119119119	-4.44796654431455\\
0.12012012012012	-5.12777658008614\\
0.146146146146146	-5.12777658008614\\
0.147147147147147	-5.74007525951299\\
0.152152152152152	-5.74007525951299\\
0.153153153153153	-12.5362788239431\\
0.2002002002002	-12.5362788239431\\
0.201201201201201	-1.08444552448558\\
0.211211211211211	-1.08444552448558\\
0.212212212212212	-0.506988989989414\\
0.253253253253253	-0.506988989989414\\
0.254254254254254	2.79312141185089\\
0.315315315315315	2.79312141185089\\
0.316316316316316	-0.0687269461230024\\
0.369369369369369	-0.0687269461230024\\
0.37037037037037	-0.126944876549684\\
0.382382382382382	-0.126944876549684\\
0.383383383383383	-5.59198535649894\\
0.448448448448448	-5.59198535649894\\
0.449449449449449	-3.99215517370783\\
0.451451451451451	-3.99215517370783\\
0.452452452452452	0.76815537890766\\
0.459459459459459	0.76815537890766\\
0.46046046046046	3.23528036778794\\
0.523523523523523	3.23528036778794\\
0.524524524524524	7.58368711003963\\
0.53953953953954	7.58368711003963\\
0.540540540540541	7.81471952908762\\
0.621621621621622	7.81471952908762\\
0.622622622622623	6.81657346392052\\
0.623623623623624	6.81657346392052\\
0.624624624624625	3.44780870815387\\
0.631631631631632	3.44780870815387\\
0.632632632632633	2.36232718500623\\
0.648648648648649	2.36232718500623\\
0.64964964964965	1.30145517601352\\
0.678678678678679	1.30145517601352\\
0.67967967967968	1.21203845359852\\
0.770770770770771	1.21203845359852\\
0.771771771771772	1.26983635990426\\
0.776776776776777	1.26983635990426\\
0.777777777777778	1.88800026778653\\
0.781781781781782	1.88800026778653\\
0.782782782782783	6.17426928064825\\
0.7997997997998	6.17426928064825\\
0.800800800800801	6.70481795829099\\
0.835835835835836	6.70481795829099\\
0.836836836836837	6.8494112434135\\
0.900900900900901	6.8494112434135\\
0.901901901901902	1.22449229510369\\
0.904904904904905	1.22449229510369\\
0.905905905905906	-3.07388395228273\\
1	-3.07388395228273\\
};
\end{axis}
\end{tikzpicture}%
%
}

\begin{abstract}
We introduce subgradient-based Lavrentiev regularisation of the form
\begin{equation*}
\mathcal{A}(u) + \alpha \partial \mathcal{R}(u) \ni f^\delta
\end{equation*}
for linear and nonlinear ill-posed problems with monotone operators~\({ \mathcal{A} }\) and general regularisation functionals~\({ \mathcal{R} }\).
In contrast to Tikhonov regularisation,
this approach perturbs the equation itself and avoids the use of the adjoint of the derivative of~\({ \mathcal{A} }\).
It is therefore especially suitable for time-causal problems that only depend on information in the past and allows for real-time computation of regularised solutions.
We establish a general well-posedness theory in Banach spaces and prove convergence-rate results with variational source conditions.
Furthermore, we demonstrate its application in total-variation denoising in linear Volterra integral operators of the first kind and parameter-identification problems in semilinear parabolic~PDEs.
\end{abstract}


\section{Introduction}
\label{sec:intro}

In this paper we study stable approximations of a class of monotone and possibly nonlinear ill-posed equations of the form
\begin{equation} \label{eq:inverse-problem}
\mathcal{A} u = f^\dagger
\end{equation}
in the presence of noisy data~\({ f^\delta }\) satisfying
\begin{equation*}
\norm{ f^\dagger - f^\delta } \leq \delta,
\end{equation*}
where~\({ \delta > 0 }\) denotes the noise level. 
Here \(\mathcal{A} \colon \X \to \X^*\) is an operator mapping a reflexive Banach space~\(\X\) into its dual~\(\X^*\),
and for simplicity we usually write \({\mathcal{A}u}\) instead of~\({ \mathcal{A}(u) }\).

A common approach to stabilising~\eqref{eq:inverse-problem} with noisy data is Tikhonov regularisation (see for instance~\autocite{EngHanNeu96,SchGraGro09}),
which amounts to minimising the functional
\begin{equation} \label{eq:tikhonov-reg}
\mathcal{T}(u) \coloneqq \tfrac{ 1 }{ 2 }\norm{ \mathcal{A}u - f^\delta }^{ 2 } + \alpha \mathcal{R}(u)
\end{equation}
over~\(\X\).
The first term guarantees that the minimisers~\({ u_\alpha^\delta }\) approximately solve~\eqref{eq:inverse-problem},
\({ \mathcal{R} \colon \X \to [0, \infty] }\) is a suitable lower-semicontinuous and convex regularisation functional that encodes qualitative a~priori information about the solution,
and \({ \alpha > 0 }\) is the regularisation parameter balancing the two terms.

One classical example of ill-posed inverse problems is the solution of linear Volterra equations of the first kind~\autocite{Gro84}.
Here we are given some real-valued, time-dependent function~\({ f^\dagger \in \textnormal{L}^2(0, T) }\) supported on an interval~\({ [0, T] }\),
and a linear Volterra operator
\begin{equation} \label{eq:volterra-operator}
\mathcal{A}u(t) \coloneqq \int_{ 0 }^{ t } k(t, s) \, u(s) \dee s,
\end{equation}
with \({ t \leq T }\) and square-integrable kernel~\({ k \in \textnormal{L}^2( [0, T]^2) }\).
If we have the prior information that the true underlying signal~\({ u^\dagger }\) in~\eqref{eq:inverse-problem} is piecewise constant,
it makes sense to apply total-variation regularisation,
in which case the regularisation functional \(\mathcal{R}\colon \textnormal{L}^2(0,T) \to [0,\infty]\) is the total variation (TV) of the distributional derivative of~\({ u }\) over~\({ [0, T] }\), that is,
\begin{equation} \label{eq:total-variation}
\mathcal{R}(u) \coloneqq \seminorm{u}_{ \textnormal{TV}} \coloneqq \abs{ \textnormal{D}u }(0, T),
\end{equation}
as discussed in more detail in \cref{sec:volterra}.
This type of regularisation has for instance been studied in a related, two-dimensional setting~\autocite{RudOshFat1992a,AcaVog94}.

Although this choice of \({ \mathcal{R} }\) yields a nonsmooth optimisation problem,
the additional effort compared to quadratic regularisation is often worthwhile because of the improved reconstructions that are structurally much closer to the true solution.
Classical regularisation with a Sobolev regularisation term \({ \mathcal{R}(u) = \norm{ u' }_{ \textnormal{L}^2 } ^2 }\) will lead to smooth regularised solutions,
whereas the TV term is capable of reconstructing discontinuities.
In addition, the recent advances in nonsmooth optimisation have led to the development of efficient numerical algorithms (see for instance~\autocite{ComPes11,KomPes15}),
which make nonquadratic regularisers feasible for practical problems.

A particularly efficient algorithm is available for TV regularisation in the case where \({ \mathcal{A} }\) is the identity operator on~\({ \textnormal{L}^2(0, T) }\) (which is not of type~\eqref{eq:volterra-operator}).
Here, Tikhonov's method~\eqref{eq:tikhonov-reg} with TV~\eqref{eq:total-variation} can be seen as a spline-based nonlinear regression method~\autocite{MamGee1997a,SteDidNeu2005a}.
Moreover, the (nonlinear) problem of minimising the Tikhonov functional~\eqref{eq:tikhonov-reg} can in this case be solved efficiently by means of the \emph{taut-string algorithm}~\autocite{DavKov2001a,Gra2007a,Con2013a},
which is a dynamic-programming method for the solution of a dual problem.

In the case of a Volterra operator of the form~\eqref{eq:volterra-operator}, the situation is different.
There, standard optimisation theory implies that the minimisers of the Tikhonov functional are precisely the solutions of the variational inclusion
\begin{equation} \label{eq:tikhonov-optconds}
0 \in \partial \mathcal{T}(u_\alpha^\delta) = \mathcal{A}^* (\mathcal{A}u_\alpha^\delta - f^\delta) + \alpha \partial \mathcal{R}(u_\alpha^\delta),
\end{equation}
where~\({ \partial }\) denotes the subdifferential.
This equation, however, requires the adjoint of \(\mathcal{A}\), which reverses the time-causal structure.
That is, whereas the operator~\(\mathcal{A}\) consists of an integration backward in time,
its adjoint~\(\mathcal{A}^*\) consists in the integration forward in time,
\begin{equation*}
\mathcal{A}^*u(t) = \int_{ t }^{ T } k(s, t) \, u(s) \dee s,
\end{equation*}
which at each point depends completely on future information.
As a consequence, Tikhonov regularisation is not applicable in situations where the input data consists of a continuous data stream and one tries to solve the inverse problem \emph{on-the-fly} in real time.
In addition, it is impossible to implement the same dynamic-programming idea as in the taut-string algorithm in order to obtain a similarly efficient numerical solution method.

An alternative to Tikhonov regularisation for the stable solution of~\eqref{eq:inverse-problem} is \emph{Lavrentiev regularisation} (see for instance~\autocite{Tau2002a,NaiTau2004a,AlbRya2006a,ArgChoGeo2013a,MahNai2013a,BotHof2016a,HofKalRes2016a,PlaMatHof2018a,GeoSre2018a,PlaHof2019b}) consisting in solving the operator equation
\begin{equation*} 
\mathcal{A}u + \alpha u = f^\delta
\end{equation*}
in the Hilbert-space setting, or more generally
\begin{equation} \label{eq:lavrentiev-reg}
\mathcal{A}u + \alpha J(u) \ni f^\delta
\end{equation}
in the Banach-space equivalent,
where \({ J \coloneqq \partial \bigl( \tfrac{ 1 }{ 2 } \norm{  }^{ 2 }  \bigr) }\) is the normalised duality mapping~\eqref{eq:norm-duality-mapping}.
This regularisation method is applicable for \emph{monotone} operators, that is, for those \({ \mathcal{A} }\) that satisfy
\begin{equation*}
\dualitypairing{ \mathcal{A}u - \mathcal{A}v }{ u - v } \geq 0 \qquad \text{for all } u, v \in \X.
\end{equation*}
In contrast to Tikhonov regularisation~\eqref{eq:tikhonov-reg},
Lavrentiev regularisation~\eqref{eq:lavrentiev-reg} simply perturbs the equation itself, has a local nature, and avoids the introduction of noncausal effects in the form of adjoints.

Suggestively, we now propose to regularise~\eqref{eq:inverse-problem} by solving the variational inclusion
\begin{equation} \label{eq:subgradient-lavrentiev-reg}
\mathcal{A}u + \alpha \partial \mathcal{R}(u) \ni f,
\end{equation}
which combines the general regularisation functionals in the Tikhonov approach with the simplicity of Lavrentiev regularisation.
One may thus view~\eqref{eq:subgradient-lavrentiev-reg} either as the optimality conditions~\eqref{eq:tikhonov-optconds} with the adjoint removed or the inclusion~\eqref{eq:lavrentiev-reg} for more general subgradients.
Observe in passing that we display~\eqref{eq:subgradient-lavrentiev-reg} for a generic~\({ f \in \X^* }\) (say, \({ f^\dagger }\) or~\({ f^\delta }\)) in order to accommodate the upcoming analysis.

In this paper, we show that this \emph{subgradient-based Lavrentiev regularisation} -- by solving the inclusion problem~\eqref{eq:subgradient-lavrentiev-reg} -- yields a completely new, well-posed regularisation method for a general class of monotone operators~\({ \mathcal{A} }\) and regularisers~\({ \mathcal{R} }\) in reflexive Banach spaces.
To this end, it is important to emphasise that our setting does not require coercivity of either~\(\mathcal{A}\) or~\(\mathcal{R}\),
as will be discussed in more detail in \cref{sec:results}.

In addition to showing well-posedness,
we derive asymptotic error estimates under the assumption of variational source conditions.
There, we follow the ideas of~\autocite{HofKalRes2016a},
where the same approach was used for the analysis of classical Lavrentiev regularisation.
All these results as well as the requirements for \(\mathcal{A}\) and \(\mathcal{R}\) are collected in \cref{sec:results}.
Moreover, \cref{sec:background} provides a review of the functional-analytic background that is required for the well-posedness proof.
We discuss a quantitative formulation of the Browder--Minty theorem concerning existence and uniqueness of monotone operator equations that includes an explicit bound on the solution.
We then provide the proof of well-posedness in \cref{sec:wellposedness} and the derivation of convergence rates in \cref{sec:convergence-rates}.
Finally, we demonstrate the new approach~\eqref{eq:subgradient-lavrentiev-reg} with TV denoising in both convolutional Volterra integral equations in \cref{sec:volterra}
and nonlinear parameter-identification problems for parabolic~PDEs in \cref{sec:param-id-problem}.
This includes new conditions on the convolution kernel~\eqref{eq:volterra-operator} that guarantee strict monotonicity of~\({ \mathcal{A} }\),
and a detailed analysis of coercivity in the PDE problem.
We also briefly present numerical results based on a generalisation of the taut-string method,
which will be the subject of a forthcoming paper.

\section{Main results}
\label{sec:results}

\subsection{Assumptions}

We study subgradient-based Lavrentiev regularisation~\eqref{eq:subgradient-lavrentiev-reg} under the following conditions and refer to \cref{sec:background} for functional-analytic definitions.

\begin{assumption} \label{assumptions}
\leavevmode
\begin{enumerate}[ref={\roman*)}]
\item \label{assumptions:space}
\textcolor{artcolor}{\textit{Space:}} \({ (\X, \norm{  }) }\) is a real, reflexive Banach space with topological dual~\({ (\X^*, \norm{  }_{ \X^* }) }\), and \({ \dualitypairing{  }{  } \coloneqq \dualitypairing{  }{  }_{\X^* \times \X} }\) denotes the duality pairing.
\item \label{assumptions:operator}
\textcolor{artcolor}{\textit{Operator:}} \({\mathcal{A} \colon \X \to \X^* }\) is strictly monotone and hemicontinuous.
\item \label{assumptions:regulariser}
\textcolor{artcolor}{\textit{Regulariser:}} \({\mathcal{R} \colon \X \to (- \infty, +\infty]}\) is proper, convex, and lower semicontinuous.
\item
\textcolor{artcolor}{\textit{Solution:}} There exists a solution~\({u^\dagger \in \domain \mathcal{R} }\) to the original problem~\eqref{eq:inverse-problem}, that is, \({ \mathcal{A}u^\dagger = f^\dagger }\).
\item \label{assumptions:coercivity}
  \textcolor{artcolor}{\textit{Coercivity:}} The sublevel set
  \({ \Set[\big]{ u \in \X \given \norm{ u } + \mathcal{R}(u) \leq C } }\)
  is compact for all~\({ C > 0 }\).
\item \label{assumptions:growth}
  \textcolor{artcolor}{\textit{Growth:}}
  For all (sufficiently large) \({ C > 0 }\) and some \({ t_0 > 0 }\) we have that
  \begin{equation}\label{eq:coercivity-growth}
    \adjustlimits \lim_{r \to \infty} \inf_{u\in U_C} \dualitypairing*{ \mathcal{A}\bigl(ru + (1-r)u^\dagger \bigr) }{ u - u^\dagger } = \infty,
  \end{equation}
  where \({ U_C \coloneqq \Set[\big]{ u \in \X \given \norm{u-u^\dagger} = t_0 \text{ and } \mathcal{R}(u) \leq C} }\).
\end{enumerate}
\end{assumption}

Note that both \({ u^\dagger }\) and the regularised solutions are necessarily unique since \({ \mathcal{A} }\) is injective by strict monotonicity.
Moreover, if \({ \mathcal{A} }\) is linear, then (hemi)continuity comes for free,
since linear monotone operators are bounded.
As regards the natural assumption \({ u^\dagger \in \domain \mathcal{R} }\),
we point out that for classical Lavrentiev regularisation it is intrinsically satisfied and means simply that \({ u^\dagger }\) is an element of~\({ \X }\).

Switching focus to \Cref{assumptions}~\ref{assumptions:coercivity}, we emphasise that \({ \domain \mathcal{R} }\) is not necessarily a compact subset of~\({ \X }\),
which otherwise easily would have guaranteed coercivity.
In fact, in the typical case when \({ \mathcal{R} }\) is the seminorm of a (not necessarily compact) subspace of~\({ \X }\),
such as when \({ \mathcal{R} }\) equals total variation or a Sobolev seminorm in Lebesgue spaces (see \cref{sec:volterra}),
then \({ \domain \mathcal{R} }\) is \emph{not} compact in~\({ \X }\).
As such, \Cref{assumptions}~\ref{assumptions:coercivity} is weaker than assuming the compactness of~\({ \domain \mathcal{R} }\).

Furthermore, the growth condition in~\Cref{assumptions}~\ref{assumptions:growth} holds trivially for strictly monotone linear~\({ \mathcal{A} }\) and shows up naturally in the proof of well-posedness.
On the practical side, we also prove directly that it is satisfied in the nonlinear parameter-identification problem in \cref{sec:param-id-problem}.
To aid further applications, we include in addition an alternative condition which may be easier to verify.

\begin{proposition}[Alternative growth conditions]
  \label{thm:alternative-growth-conditions}
  Assume that \Cref{assumptions}~\ref{assumptions:space}--\ref{assumptions:regulariser}
  and \ref{assumptions:coercivity} is satisfied.
  Let $u_{\textnormal{ref}}  \in \X$ be a fixed reference point
  and assume that there exists a function
  \({ \gamma \colon (0, \infty) \to (0, \infty) }\)
  with superlinear growth,
  \[
    \lim_{r\to \infty} \frac{\gamma(r)}{r} = \infty,
  \]
  such that the \enquote{spherical growth condition}
  \begin{equation}
    \sup_{ \norm{ u } = 1 } \frac{ \norm*{ \mathcal{A}\bigl(ru + (1-r)u_{\textnormal{ref}}\bigr) }_{ \X^* } }{ \gamma(r) } \xrightarrow[ r \to\infty ]{  } 0  \label{eq:spherical-growth}
  \end{equation}
  holds and the \enquote{directional growth condition}
  \begin{equation}
    \dualitypairing*{ \mathcal{A}\bigl(ru + (1-r)u_{\textnormal{ref}}\bigr) - \mathcal{A}u_{ \textnormal{ref}} }{ u - u_{ \textnormal{ref}}} \geq \frac{\gamma(r)}{r} \dualitypairing*{ \mathcal{A}u - \mathcal{A}u_{ \textnormal{ref}} }{ u - u_{ \textnormal{ref}}} \label{eq:directional-growth}
  \end{equation}
  holds for all \({ u \in \X }\) with \({ \norm{ u - u_{\textnormal{ref}}} = 1 }\) and sufficiently large~\({ r > 0 }\).
  Then the growth condition in \Cref{assumptions}~\ref{assumptions:growth} is satisfied.

  In particular, this applies with \({ u_{\textnormal{ref}} = 0 }\)
  and \({\gamma(r) = r^2 }\) in the case where \({ \mathcal{A} }\)
  is a strictly monotone linear operator.
\end{proposition}

The proof of \cref{thm:alternative-growth-conditions} can be found in \cref{sec:alternative-growth-conditions}.
We also note that condition~\eqref{eq:spherical-growth} is superfluous if we can choose~\({ u_{ \textnormal{ref}} = u^\dagger }\);
for details, we refer to \cref{sec:alternative-growth-conditions}.

Observe that~\eqref{eq:directional-growth} is a relatively mild condition, which implies a form of \enquote{weak directional coercivity} of \({ \mathcal{A} }\) relative to the point~\({ u_{ \textnormal{ref}} }\)
in the sense that
\begin{equation} \label{eq:weak-directional-coercivity}
  \dualitypairing*{ \mathcal{A}(u_{\textnormal{ref}} + rv) - \mathcal{A}u_{\textnormal{ref}}}{ v} \xrightarrow[ r \to \infty ]{ } \infty
\end{equation}
for all \({ v \in \X }\) with \({ \norm{ v } = 1 }\).
Since the limit~\eqref{eq:weak-directional-coercivity} is not necessarily uniform over the unit sphere,
that~is, \({ r }\)~may be a function of~\({ v }\)---or there might exist a sequence of \enquote{bad}~\({ u }\)'s for which the last factor
\begin{equation*}
\dualitypairing{ \mathcal{A}u - \mathcal{A}u_{ \textnormal{ref}} }{ u - u_{ \textnormal{ref}} }
\end{equation*}
goes to~\({ 0 }\) sufficiently fast in~\eqref{eq:directional-growth}---this neither leads to coercivity of \({ \mathcal{A}  }\) nor \({ \mathcal{A} + \alpha \partial \mathcal{R} }\) of itself.
The compactness of the sublevel set in \Cref{assumptions}~\ref{assumptions:coercivity} together with strict monotonicity and~\eqref{eq:spherical-growth}, however, ensure that the possibly bad sequence of \({ u }\)'s is not an issue---even when the noise level~\({ \delta }\) and the regularisation parameter~\({ \alpha }\) go to~\({ 0 }\) appropriately.
The interpretation of~\eqref{eq:coercivity-growth} in light of~\eqref{eq:weak-directional-coercivity} might thus be labeled as weak directional coercivity relative to~\({ u^\dagger }\) uniformly over~\({ U_C }\).

\subsection{Well-posedness}

We obtain well-posedness of subgradient-based Lavrentiev regularisation~\eqref{eq:subgradient-lavrentiev-reg} under \Cref{assumptions} in the following manner.
Remember that \({ f^\dagger }\) is the exact data for which \({ \mathcal{A}u^\dagger = f^\dagger }\),
that \({ f^\delta }\) is its noisy counterpart,
and that \({ f }\) refers to any element of~\({ \X^* }\).

\begin{theorem}[Existence] \label{thm:existence}
For every~\({f \in \X^* }\) and~\({\alpha > 0}\),
there exists a unique solution~\({u_\alpha \in \domain \mathcal{R} }\) to the regularised problem~\eqref{eq:subgradient-lavrentiev-reg}.
\end{theorem}

\begin{theorem}[Stability] \label{thm:stability}
For every \({ f \in \X^* }\) and fixed~\({\alpha > 0}\),
if \({f_k \to f}\) in~\({ \X^* }\), then
\begin{equation*}
u_{\alpha, k} \to u_\alpha \quad \text{in~\({ \X }\)}, \qquad \mathcal{A}u_{\alpha, k} \rightharpoonup \mathcal{A}u_\alpha \quad \text{weakly in~\({ \X^* }\),} \qquad \text{and} \qquad \mathcal{R}(u_{\alpha, k}) \to \mathcal{R}(u_\alpha),
\end{equation*}
where~\({u_{\alpha, k}}\) and \({ u_\alpha }\) solve~\eqref{eq:subgradient-lavrentiev-reg} with right-hand sides~\({ f_k }\) and~\({ f }\), respectively.
\end{theorem}

\begin{theorem}[Convergence] \label{thm:convergence}
Let \({\alpha = \alpha(\delta)}\) be any a~priori parameter choice for which
\begin{equation} \label{eq:convergence-apriori}
\alpha \to 0 \qquad \text{and} \qquad \frac{ \delta }{ \alpha } \quad \text{is bounded} \qquad \text{as } \delta \to 0,
\end{equation}
and denote by \({u_\alpha^\delta \coloneqq u_{ \alpha(\delta)}^\delta }\) the regularised solutions of~\eqref{eq:subgradient-lavrentiev-reg} with \({ f^\delta \in \X^*}\) satisfying \({\norm{f^{ \delta} - f^\dagger}_{\X^*} \leq \delta}\).
Then
\begin{equation*}
u_\alpha^\delta \to u^\dagger \quad \text{in~\({ \X }\)}, \qquad \mathcal{A}u_\alpha^\delta \rightharpoonup f^\dagger \quad \text{weakly in~\({ \X^* }\),}  \qquad \text{and} \qquad \mathcal{R}(u_\alpha^\delta) \to \mathcal{R}(u^\dagger)
\end{equation*}
as \({ \delta \to 0 }\).
\end{theorem}

We wish to draw attention to the unusually weak relationship between \({ \delta }\) and \({ \alpha }\) in \cref{thm:convergence} in that merely boundedness of \({ \delta / \alpha }\) suffices for convergence.
This contrasts results both for Tikhonov-type regularisations
\begin{equation*}
u_\alpha^\delta = \argmin_{ u \in \X } \left( \norm{ \mathcal{A}u - f^\delta }_{  \X^* }^{ p } + \alpha \mathcal{R}(u) \right)
\end{equation*}
for \({ p > 1 }\),
where one has to require that \({ \delta^p / \alpha \to 0 }\) as \({ \alpha(\delta) \to 0 }\);
see for instance~\autocite[Proposition~3.1]{HofYam2010a},
and for standard Lavrentiev regularisation~\eqref{eq:lavrentiev-reg},
in which the condition \({ \delta/ \alpha \to 0 }\) is needed to guarantee convergence~\autocite[Theorem~2.2.4]{AlbRya2006a},~\autocite[Proposition~1]{HofKalRes2016a}.
It should still be noted that one has \emph{weak} convergence in~\({ \X }\) under the parameter choice~\eqref{eq:convergence-apriori} for Lavrentiev regularisation~\autocite[Theorems~2.1.6 and~2.2.6]{AlbRya2006a},
which rationalises our need for a compactness condition in \Cref{assumptions}~\ref{assumptions:coercivity} in order to obtain strong convergence.

\begin{remark}
If~\({ \mathcal{A} }\) is linear,
then weak convergence in the codomain~\({ \X^* }\) in \cref{thm:stability,thm:convergence} always becomes strong convergence,
because linear monotone maps are continuous.

We can also upgrade weak convergence \({ \mathcal{A}u_\alpha^\delta \rightharpoonup \mathcal{A}u^\dagger }\) in \cref{thm:convergence} to strong convergence if \({ u^\dagger }\) is a continuity point of~\({ \mathcal{R} }\)
(which is the case if \({ \mathcal{R} = \tfrac{ 1 }{ 2 } \norm{  }^2 }\) in classical Lavrentiev regularisation in Hilbert spaces).
In this setting, \({ \Set[\big]{ u_\alpha^\delta }{}_\delta^{} }\) lies in the interior of~\({ \domain \mathcal{R} }\),
ignoring a finite number of \({ u_\alpha^\delta }\)'s if necessary.
By the Rockafellar--Veselý theorem~\autocite[Theorem~1]{Roc1969a},
\({ \partial \mathcal{R} }\) is locally bounded in its interior domain---meaning that for every \({ u \in \interior (\domain \mathcal{R}) }\) there exists an \({ \epsilon > 0 }\) such that the image of an \({ \epsilon }\)-ball around~\({ u }\) is bounded.
This implies that \({ \partial \mathcal{R}\bigl( \Set[\big]{ u_\alpha^\delta }{}_\delta^{}\bigr) }\) is bounded by compactness of~\({ \Set[\big]{ u_\alpha^\delta }{}_\delta^{} }\).
In particular, the elements \({ \xi_\alpha^\delta \in \partial \mathcal{R}(u_\alpha^\delta) }\) satisfying \({ \mathcal{A}u_\alpha^\delta + \alpha \xi_\alpha^\delta = f^\delta }\) are uniformly bounded in~\({ \X^* }\) as \({ \delta }\) and \({ \alpha }\) go to~\({ 0 }\),
from which it follows that~\({ \mathcal{A}u_\alpha^\delta \to f^\dagger }\).
\end{remark}

Finally, we mention that it may sometimes be beneficial to consider the modified variant
\begin{equation*}
\mathcal{A}u + \alpha \partial \mathcal{R}(u - u_{ \textnormal{init}} ) \ni f^\delta
\end{equation*}
of the regularisation method~\eqref{eq:subgradient-lavrentiev-reg} in which \({ u_{ \textnormal{init}} \in \X }\) denotes an initial guess.
If we assume that \({ (u^\dagger - u_{ \textnormal{init}} ) \in \domain \mathcal{R} }\),
all the results and proofs carry over to this case,
with the only difference being that
\begin{align*} \SwapAboveDisplaySkip
\mathcal{R}(u_{\alpha, k} - u_{ \textnormal{init}} ) \to \mathcal{R}(u_\alpha - u_{ \textnormal{init}} ) &\qquad \text{replaces} \qquad \mathcal{R}(u_{\alpha, k}) \to \mathcal{R}(u_\alpha)
\intertext{in \cref{thm:stability}, and in \cref{thm:convergence} that}
\mathcal{R}(u_\alpha^\delta - u_{ \textnormal{init}} ) \to \mathcal{R}(u^\dagger - u_{ \textnormal{init}} ) &\qquad \text{replaces} \qquad \mathcal{R}(u_\alpha^\delta) \to \mathcal{R}(u^\dagger).
\end{align*}

\subsection{Convergence rates}

As regards quantitative error estimates for~\eqref{eq:subgradient-lavrentiev-reg},
we use the method of \emph{variational source conditions} established in \autocite{HofKalPosSch2007a,Gra2010a} and generalised in \autocite{HofKalRes2016a} to classical Lavrentiev regularisation.
That is, we consider general variational source conditions of the form
\begin{equation}\label{eq:varineq}
D(u, u^\dagger) \leq \mathcal{R}(u) - \mathcal{R}(u^\dagger) + \varphi\bigl(\dualitypairing{ \mathcal{A}u - \mathcal{A}u^\dagger }{ u - u^\dagger }  \bigr) \qquad \text{for all } u \in \mathcal{M},
\end{equation}
where \({ D \colon \X \times \X \to [0, \infty] }\) is any distance-like function and \({ \varphi \colon [0, \infty) \to [0, \infty) }\) denotes an index map; see \cref{sec:convergence-rates}.
The set \({ \mathcal{M} \subseteq \X }\), typically a ball, must contain all regularised solutions~\({ u_\alpha^\delta }\) of interest for sufficiently small \({ \delta > 0 }\) and \({ \alpha > 0 }\) chosen appropriately.
A~priori we do not assume any further properties other than nonnegativity of~\({ D }\).
This gives the following main result,
where \({ \psi }\) denotes the convex (Fenchel) conjugate of~\({ \varphi^{-1} }\) (defined in \cref{sec:convergence-rates}).
Notationally, we write \({ A \lesssim B }\) when \({ A \leq c B }\) for an independent constant~\({ c > 0 }\),
and \({ A \sim B }\) means that~\({ A \lesssim B \lesssim A }\).

\begin{theorem}[General convergence rates] \label{thm:rates-main}
Assume that the variational source condition~\eqref{eq:varineq} holds and that the ratio~\({ \delta / \alpha }\) is uniformly bounded.
Then there exists a constant~\({ C > 0 }\) such that
\begin{equation*}
D(u_\alpha^\delta, u^\dagger) \leq C \frac{\delta}{\alpha} + \frac{\psi(\alpha)}{\alpha}
\end{equation*}
whenever \({ \delta > 0 }\) and~\({ \alpha > 0 }\) are sufficiently small.
In particular, a parameter choice \({ \alpha \sim \psi^{-1}(\delta) }\) gives the convergence rate
\begin{equation*}
D(u_\alpha^\delta, u^\dagger) \lesssim \frac{ \delta }{ \psi^{-1}(\delta) }
\end{equation*}
as~\({ \delta \to 0 }\).
\end{theorem}

In practice, one considers distance measures~\({ D }\) like the Bregman distance with respect to the functional~\({ \mathcal{R} }\) or powers of the norm on~\({ \X }\).
Specialising to the latter case and with \({ \varphi }\) being of Hölder type in~\eqref{eq:varineq},
the estimates and convergence rates can be further improved, leading to the following result.

\enlargethispage{\baselineskip} 

\begin{theorem}\label{thm:rates-norm}
Assume that the Hölder-type variational source condition
\begin{equation}\label{eq:varineq-norm}
c_1\norm{u-u^\dagger}^r \leq \mathcal{R}(u) - \mathcal{R}(u^\dagger) + c_2\left(\dualitypairing{ \mathcal{A}u - \mathcal{A}u^\dagger }{ u - u^\dagger }  \right)^{1/p}
\end{equation}
holds for some \({ r > 1 }\), \({ p > 1 }\), and \({ c_1, c_2 > 0 }\).
Then the parameter choice
\begin{equation*}
  \alpha \sim \delta \tothepowerof \left(\frac{r(p - 1)}{rp - 1}\right)
\end{equation*}
gives the convergence rate
\begin{equation*}
\norm{u_\alpha^\delta-u^\dagger} \lesssim \delta \tothepowerof \left( \frac{ 1 }{ rp - 1 } \right)
\end{equation*}
as~\({ \delta \to 0 }\).
\end{theorem}
In particular, when \({ r = 2 = p }\) in~\eqref{eq:varineq-norm}, we deduce from the variational source condition
\begin{equation*}
\norm{u-u^\dagger}^2 \lesssim \mathcal{R}(u) - \mathcal{R}(u^\dagger) + \bigl( \dualitypairing{ \mathcal{A}u - \mathcal{A}u^\dagger }{ u - u^\dagger }  \bigr)^{1/2}
\end{equation*}
a convergence rate
\begin{equation*}
\norm{u_\alpha^\delta - u^\dagger} \lesssim \delta^{1/3} \qquad \text{if } \alpha \sim \delta^{2/3},
\end{equation*}
which is in line with expected results for Lavrentiev regularisation~\autocite[Theorem~5]{HofKalRes2016a}.
In fact, \cref{thm:rates-norm} generalises the Hölder rates in~\autocite[Section~3.3]{HofKalRes2016a} from \({ r = 2 }\) to~\({ r > 1 }\),
and it should be noted that all the convergence-rates results in our work applies to the classical case \({ \mathcal{R} = \tfrac{ 1 }{ 2 } \norm{  }^{ 2 }  }\) as well.
We also point out that if \({ \mathcal{A} }\) is positively homogeneous---for example,
if \({ \mathcal{A} }\) is linear---then \cref{thm:rates-norm} is realistic only for~\({ p \geq 2 }\);
see the discussion in \cref{rem:rates}.

\section{Functional-analytic background}
\label{sec:background}

We refer to~\autocite{Sho1997a,Zei1990a,AlbRya2006a,BauCom2017a,Bar2010a} for in-depth analysis and extensions of the theory presented freely in this section,
noting that \autocite{AlbRya2006a}~thoroughly covers regularisation of ill-posed problems and that~\autocite{BauCom2017a} focuses on the Hilbert-space setting.

Our main tool in the study of~\eqref{eq:subgradient-lavrentiev-reg} is the fundamental theorem in monotone operator theory.

\begin{theorem}[Browder--Minty~\protect{\autocite{Bro1968a}}] \label{thm:browder-minty}
Maximally monotone, coercive operators \({ \X \rightrightarrows \X^* }\) are surjective.
\end{theorem}
This result generalises the fact that continuous functions on~\({ \R }\) which are monotone and coercive, also are surjective.
A~multivalued map \({ \mathcal{A} \colon \X \rightrightarrows \X^* }\) is \emph{monotone} on its domain \({ \domain \mathcal{A} \coloneqq \Set{ u \in \X \given \mathcal{A}u \neq \emptyset} }\) if
\begin{equation} \label{eq:monotone-operator}
\dualitypairing{ f - g }{ u - v } \geq 0
\end{equation}
for all \({ u, v \in \domain \mathcal{A} }\) and~\({ f \in \mathcal{A}u}\) and~\({ g \in \mathcal{A}v }\),
where we usually write \({ \mathcal{A}u }\) instead of \({ \mathcal{A}(u) }\) even though \({ \mathcal{A} }\) may be nonlinear.
In particular, a~single-valued operator is monotone if \({ \dualitypairing{ \mathcal{A}u - \mathcal{A}v }{ u - v } \geq 0
 }\) for all \({ u, v \in \domain \mathcal{A} \coloneqq \Set{ u \in \X \given \mathcal{A}u \in \X^* } }\).
Moreover, \({ \mathcal{A} }\) is \emph{strictly monotone} if equality in~\eqref{eq:monotone-operator} only holds when \({ u = v }\),
in which case it follows that \({ \mathcal{A} }\) is injective on this set.
\({ \mathcal{A} }\)~is further said to be \emph{maximally monotone} if it has no proper monotone extension---that is,
if \({ \dualitypairing{ f - g }{ u - v } \geq 0 }\) for all~\({ (v, g) \in \mathcal{A} }\),
then~\({ (u, f) \in \mathcal{A} }\), where we have identified \({ \mathcal{A} }\) with its graph.

Of particular interest is the \emph{subdifferential} \({ \partial \mathcal{R} \colon \X \rightrightarrows \X^* }\) of a convex function \({ \mathcal{R} \colon \X \to (-\infty, \infty] }\) defined by
\begin{equation*}
\xi \in \partial \mathcal{R}(u) \quad \Longleftrightarrow \quad \mathcal{R}(v) - \mathcal{R}(u) \geq \dualitypairing{ \xi }{ v - u } \quad \text{for all } v \in \X,
\end{equation*}
which is readily seen to be monotone.
If~\({ \mathcal{R} }\) additionally is lower semicontinuous and proper,
meaning that \({ \domain \mathcal{R} \coloneqq \Set{ u \in \X \given \mathcal{R}(u) < \infty } }\) is nonempty,
then the Moreau--Rockafellar theorem~\autocite[Theorem~1]{Ras2021a} shows that \({ \partial \mathcal{R} }\) is maximally monotone and proper.
A~special example is the \emph{normalised duality mapping}
\begin{equation} \label{eq:norm-duality-mapping}
J(u) \coloneqq \partial \bigl( \tfrac{ 1 }{ 2 }\norm{  }^2 \bigr) (u) = \Set*{ f \in \X^* \given \norm{ f }_{ \X^* }^{ 2 } = \dualitypairing{ f }{ u } = \norm{ u }^{ 2 }  },
\end{equation}
which is single-valued if \({ \X^* }\) is strictly convex~\autocite[Lemma~1.5.5]{AlbRya2006a}---in particular, it collapses to the identity in Hilbert spaces---and represents the regulariser in classical Lavrentiev regularisation.

Maximal monotonicity also holds for hemicontinuous monotone \({ \mathcal{A} \colon \X \to \X^* }\), where \emph{hemicontinuity} is a type of weak directional continuity in the sense that
\begin{equation*}
\dualitypairing{ \mathcal{A}(u + \epsilon v) }{ w } \xrightarrow[ \epsilon \to 0^+ ]{  } \dualitypairing{ \mathcal{A}u }{ w } 
\end{equation*}
for all~\({ u, v, w \in \X }\).
Every strong-to-weak continuous (\emph{demicontinuous}) operator---meaning that \({ u_k \to u }\) in~\({ \X }\) implies \({ \mathcal{A}u_k \rightharpoonup \mathcal{A}u }\) weakly in~\({ \X^* }\)---is hemicontinuous,
and in fact, for monotone \({ \X \to \X^* }\) both hemicontinuity, demicontinuity, and maximal monotonicity are equivalent.	
In the special case of linear~\({ \mathcal{A} }\), hemicontinuity is automatic,
and paired with monotonicity one even obtains continuity.

Monotonicity is further preserved for sums \({ \mathcal{A} + \mathcal{B} }\),
and strictly so provided at least one map is strictly monotone.
For maximally monotone operators the sum rule is more delicate,
but is at least true when one of the terms is a subgradient.
In particular, this leads to maximal monotonicity of \({ \widetilde{ \mathcal{A} } \coloneqq \mathcal{A} + \alpha \partial \mathcal{R} }\) for all~\({ \alpha > 0 }\) in~\eqref{eq:subgradient-lavrentiev-reg} under \Cref{assumptions}.

As for coercivity, we shall need a localised, quantitative form of \cref{thm:browder-minty}.
On the one hand, this allows for less restrictive assumptions on the forward operator~\({ \mathcal{A} }\),
while on the other hand, we gain an a~priori estimate of the size of the regularised solutions
which is crucial in the proofs of stability and convergence of~\eqref{eq:subgradient-lavrentiev-reg}.
Specifically, whereas the standard criterion
\begin{equation*}
\frac{ \dualitypairing{ y }{ u }  }{ \norm{ u } } \to \infty \qquad \text{for all} \quad (u, y) \in \widetilde{ \mathcal{A} } \quad \text{as} \quad  \norm{ u } \to \infty
\end{equation*}
is \emph{uniform} over all data \({ f \in \X^* }\) and centred at the origin,
we shall instead assume the existence of a point~\({ u_{ \textnormal{coerc.}} }\) and an \({ r > 0 }\) (depending on~\({ f }\)) such that
\begin{equation} \label{eq:quantitative-coercivity}
\dualitypairing{ y - f}{ u - u_{ \textnormal{coerc.}} } \geq  0 \qquad \text{for all} \quad (u, y) \in \widetilde{ \mathcal{A} } \quad \text{with} \quad \norm{ u - u_{ \textnormal{coerc.}} } \geq r.
\end{equation}
A~priori we may then conclude that every solution~\({ u_\star }\) of \({ \widetilde{ \mathcal{A} }u \ni f }\) satisfies \({ \norm{ u_\star - u_{ \textnormal{coerc.}} } \leq r }\) (see~\autocite[Theorem~1.7.9]{AlbRya2006a}),
which quantitatively relates the solutions and the data.

In fact, \cref{thm:browder-minty} adapted to~\eqref{eq:quantitative-coercivity} can now be quickly proved as follows~\autocite[along the lines of the proof of Theorem~1.7.5]{AlbRya2006a},
where we set \({ u_{ \textnormal{coerc.}} = 0 }\) for clarity.
By Minty--Rockafellar's theorem (see~\autocite[Theorem~6]{Ras2021a} for a very short argument),
stating that \({ \widetilde{ \mathcal{A} } + \lambda J }\) is surjective for every~\({ \lambda > 0 }\),
there exist \({ (u_\lambda, y_\lambda) \in \widetilde{ \mathcal{A} } }\) and \({ j_\lambda \in Ju_\lambda }\) satisfying \({ y_\lambda + \lambda j_\lambda = f }\).
Then
\begin{equation*}
\dualitypairing{ y_\lambda - f }{ u_\lambda } = - \lambda \dualitypairing{ j_\lambda }{ u_\lambda } = - \lambda \norm{ u_\lambda }^{ 2 } \leq 0
\end{equation*}
for every \({ \lambda > 0 }\).
If \({ u_\lambda = 0 }\) for some~\({ \lambda }\),
then \({ j_\lambda = 0 }\),
and so \({ y_\lambda = f }\) and \({ u_\star = 0 }\) is a solution.
Otherwise, we must have
\begin{equation*}
\dualitypairing{ y_\lambda - f }{ u_\lambda } < 0
\end{equation*}
for all \({ \lambda }\), 
from which~\eqref{eq:quantitative-coercivity} tells us that \({ \norm{ u_\lambda } < r }\) for all~\({ \lambda }\).
By reflexivity and up to a subsequence,
\({ \Set{ u_\lambda }_\lambda }\) therefore converges weakly in~\({ \X }\) to some~\({ u_\star }\) as~\({ \lambda \to 0 }\),
and since
\begin{equation*}
\dualitypairing{ f - \lambda j_\lambda - y }{ u_\lambda - u } = \dualitypairing{ y_\lambda - y }{ u_\lambda - u } \geq 0
\end{equation*}
for all \({ (u, y) \in \widetilde{ \mathcal{A} } }\) by monotonicity,
we may take limits as \({ \lambda \to 0 }\) and deduce that \({ \dualitypairing{ f - y }{ u_\star - u } \geq 0 }\),
using that~\({ J }\) is bounded.
Now maximal monotonicity gives \({ (u_\star, f ) \in \widetilde{ \mathcal{A} } }\),
and \({ \norm{ u_\star } \leq \liminf \norm{ u_\lambda } \leq r }\).

This yields an alternative route to the following result,
where we note that~\eqref{eq:quantitative-coercivity-subgradient} is a slightly more restrictive but apparently easier variant of~\eqref{eq:quantitative-coercivity} using the definition of~\({ \partial \mathcal{R}(u) }\).

\begin{theorem}[Browder~\protect{\autocite{Bro1966a}}] \label{thm:browder}
Let~\({ \mathcal{A} \colon \X \to \X^* }\) be maximally monotone and \({ \mathcal{R} \colon\X \to (- \infty, \infty] }\) be proper, convex, and lower semicontinuous.
Suppose that for a given \({ f \in \X^* }\),
there exist \({ u_{ \textnormal{coerc.}} \in \domain \mathcal{R} }\) and \({ r > 0 }\) such that
\begin{equation} \label{eq:quantitative-coercivity-subgradient}
\dualitypairing{\mathcal{A}u - f}{u - u_{ \textnormal{coerc.}}} + \mathcal{R}(u) - \mathcal{R}(u_{ \textnormal{coerc.}}) \geq 0
\end{equation}
for all \({ u \in \X }\) with~\({\norm{u-u_{\textnormal{coerc.}}} \geq r }\).
Then there exists \({ u_\star \in \domain \mathcal{R} }\) satisfying \({ \norm{ u_\star - u_{\textnormal{coerc.}} } \leq r }\) and
\begin{equation*}
\mathcal{A}u_\star + \partial \mathcal{R}(u_\star) \ni f.
\end{equation*}
\end{theorem}

\section{Proof of well-posedness}
\label{sec:wellposedness}

We now establish existence, stability, and convergence for subgradient-based Lavrentiev regularisation~\eqref{eq:subgradient-lavrentiev-reg} under \Cref{assumptions} with help of~\cref{thm:browder}. Without loss of generality we assume in this section that~\({ \mathcal{R}(u^\dagger) = 0 }\).
Recall also that \({ f^\dagger }\) is the exact data for which \({ \mathcal{A}u^\dagger = f^\dagger }\),
that \({ f^\delta }\) is its noisy counterpart,
and that \({ f }\) refers to any element of~\({ \X^* }\).

\begin{remark}
In order to obtain convergence,
it seems necessary to prove coercivity of \({ \mathcal{A} + \alpha \partial \mathcal{R} }\) relative to the exact solution~\({ u^\dagger }\) for the monotonicity arguments to work.
This is not an issue for the existence of regularised solutions for \emph{fixed}~\({ \alpha }\),
but we choose \({ u_{ \textnormal{coerc.}} = u^\dagger }\) also in this case.
Observe further that there is no direct relationship between \({ u_{ \textnormal{coerc.}} }\) in \cref{thm:browder} and \({ u_{ \textnormal{ref}} }\) from \cref{thm:alternative-growth-conditions}; \({ u_{ \textnormal{coerc.}} }\) is a point where the combination of \({ \mathcal{A} }\) and \({ \partial \mathcal{R} }\) is coercive,
whereas \({ u_{ \textnormal{ref}} }\) relates to \({ \mathcal{A} }\) alone.
\end{remark}

\subsection{Existence---proof of \cref{thm:existence}}

We apply \cref{thm:browder}, where it remains to establish~\eqref{eq:quantitative-coercivity-subgradient}.
We choose in the following \({ u_{ \textnormal{coerc.}} = u^\dagger }\).
Moreover, given \({ u \in \X }\) we denote
\[
  u_r \coloneqq ru + (1-r)u^\dagger = u^\dagger + r(u-u^\dagger).
\]
Then~\eqref{eq:quantitative-coercivity-subgradient} is equivalent to the existence of an \({ r_0 > 0 }\) such that
\begin{equation}\label{eq:existence_toshow}
  \dualitypairing*{ \mathcal{A}u_r - f }{ r(u - u^\dagger) } + \alpha \mathcal{R}(u_r) \geq 0
\end{equation}
for all \({u \in \X}\) with \({\norm{u-u^\dagger} = t_0}\) and \({r > r_0}\).
For simplicity, we assume in the following that \({t_0 = 1}\).

Define now
\begin{equation}\label{eq:regulariser-contradiction-bnd}
  C \coloneqq \frac{\norm{f^\dagger - f}_{ \X^* }}{\alpha}
\end{equation}
and let \({V \coloneqq \Set[\big]{u \in \X \given \norm{u} + \mathcal{R}(u) \leq C+\norm{u^\dagger} + 1}}\).
By \Cref{assumptions}~\ref{assumptions:coercivity}, the set \({V}\) is compact,
and thus the lower semicontinuity of \({\mathcal{R}}\) implies that
\({\mathcal{R}}\) is bounded below on~\({V}\), say, \({\mathcal{R}(u) \geq -K}\) for all~\({u \in V}\).
Next we define
\[
  r_0 \coloneqq \sup\Set[\Big]{r > 0 \given \inf_{u \in U_{C}} \dualitypairing*{\mathcal{A}u_r}{u-u^\dagger}
    \leq \alpha K  + \norm{f}_{ \X^* }
  },
\]
where \({U_{C} \coloneqq \Set[\big]{u \in \X \given \norm{u-u^\dagger} = 1 \text{ and } \mathcal{R}(u) \leq C}}\).
Such an \({r_0}\) exists in view of the growth
condition in \Cref{assumptions}~\ref{assumptions:growth}.

Assume now that \({r > r_0}\) and that \({u \in \X}\) satisfies \({\norm{u-u^\dagger} = 1}\).
By the decomposition
\[
  { \mathcal{A}u_r - f = \left(  \mathcal{A}u_r - \mathcal{A}u^\dagger \right) + \left( f^\dagger - f \right)  },
\]
it follows from monotonicity that
\begin{align*}
\dualitypairing*{ \mathcal{A}u_r - f }{ u - u^\dagger } &\geq \dualitypairing*{ \mathcal{A}u_r - \mathcal{A}u^\dagger }{ u - u^\dagger } - \norm{ f^\dagger - f }_{ \X^* } \norm{ u - u^\dagger } \\[1ex]
& \geq 0 - \norm{ f^\dagger - f }_{ \X^* } = - C\alpha,
\end{align*}
and thus
\[
  \dualitypairing*{ \mathcal{A}u_r - f }{ r(u - u^\dagger) } + \alpha \mathcal{R}(u_r)
  \geq -C\alpha r + \alpha\mathcal{R}(u_r).
\]
Therefore, if \({\mathcal{R}(u_r) > Cr}\), then~\eqref{eq:existence_toshow} holds.

Let us now assume that \({\mathcal{R}(u_r) \leq Cr}\).
The convexity of \({\mathcal{R}}\) and the assumption \({\mathcal{R}(u^\dagger) = 0}\) imply that
\begin{equation*}
r \mathcal{R} (u) = r\mathcal{R}\Bigl(\tfrac{1}{r}u_r + \tfrac{r-1}{r}u^\dagger\Bigr)\leq \mathcal{R}(u_r) + (r - 1) \mathcal{R}(u^\dagger) = \mathcal{R}(u_r) \leq Cr,
\end{equation*}
that is, \({\mathcal{R}(u) \leq C}\), which implies that \({u \in U_C}\).
Since \({r > r_0}\), it follows that
\[
  \dualitypairing*{\mathcal{A}u_r}{u-u^\dagger} > \alpha K + \norm{f}_{ \X^* }.
\]
Moreover, we have that
\[
  \norm{u} + \mathcal{R}(u) \leq \norm{u-u^\dagger} + \norm{u^\dagger} + \mathcal{R}(u) \leq 1+\norm{u^\dagger} + C,
\]
and thus \({u \in V}\), which in turn implies that \({\mathcal{R}(u) \geq -K}\).
Therefore
\begin{align*}
  \dualitypairing*{ \mathcal{A}u_r - f }{ r(u - u^\dagger) } + \alpha \mathcal{R}(u_r)
  &> \bigl(\alpha K + \norm{f}_{ \X^* } \bigr) \, r - \dualitypairing*{f}{r(u-u^\dagger)} + r\alpha\mathcal{R}(u)\\[1ex]
  &\geq \alpha Kr + r\norm{f}_{ \X^* } - r\norm{f}_{ \X^* } -r\alpha K = 0.
\end{align*}
This shows that~\eqref{eq:existence_toshow} holds for all \({u\in\X}\)
with \({\norm{u-u^\dagger}= 1 \; (= t_0) }\) and \({\mathcal{R}(u)\leq C}\), and all \({r > r_0}\), which concludes the proof.

\enlargethispage{\baselineskip} 

As a consequence of the above arguments, we note the following preliminary stability result.
\begin{corollary} \label{thm:boundedness}
Let~\({ \alpha > 0 }\) be fixed.
If~\({ \Set{ f_k }_k }\) is a bounded sequence of data in~\({ \X^* }\),
then the corresponding sequence of regularised solutions~\({\Set{u_{\alpha, k}}_k}\) solving~\({\mathcal{A}u + \alpha \partial \mathcal{R}(u) \ni f_k}\) is also bounded.
\end{corollary}
\begin{proof}
  Since \({ \Set{ f_k }_k }\) is bounded, it follows that
  \[
    \widetilde{C} \coloneqq \sup_k \frac{\norm{f^\dagger-f_k}_{ \X^* }}{\alpha} < \infty.
  \]
  Replacing the constant \({C}\) in~\eqref{eq:regulariser-contradiction-bnd} by \({\widetilde{C}}\)
  and correspondingly modifying the values of \({K}\) and~\({r_0}\),
  we obtain by following the proof of \cref{thm:existence}
  that
  \[
    \dualitypairing*{ \mathcal{A}u_r - f_k }{ r(u - u^\dagger) } + \alpha \mathcal{R}(u_r) > 0
  \]
  for all \({u \in \X}\) with \({\norm{u-u^\dagger} = 1 \; (= t_0)}\), all \({r > r_0}\), and all \({k \in \N}\).  
  As a consequence of Theorem~\ref{thm:browder},
  we therefore obtain the estimate \({\norm{u_{\alpha,k}-u^\dagger} \leq r_0}\) for all~\({k}\),
  which proves the assertion.
\end{proof}

\subsection{Stability---proof of \cref{thm:stability}}

Let \({ \alpha > 0 }\) be fixed and \({ f \in \X^* }\) be given,
and assume that \({ \Set{ f_k }_k \subseteq \X^* }\) is a sequence of data converging to~\({ f }\).
Furthermore, let \({ u_\alpha }\) and \({ u_{ \alpha, k} }\) be the respective regularised solutions of
\begin{alignat}{2} 
&\mathcal{A}u_\alpha + \alpha \partial\mathcal{R}(u_\alpha) \ni f \qquad &&\text{and} \qquad \mathcal{A}u_{ \alpha, k} + \alpha \partial\mathcal{R}(u_{ \alpha, k}) \ni f_k, \notag \\
\shortintertext{so that}
&\mathcal{A}u_\alpha + \alpha \xi_\alpha = f \qquad &&\text{and} \qquad \mathcal{A}u_{ \alpha, k} + \alpha \xi_{ \alpha, k} = f_k  \label{eq:regularized-equations}
\end{alignat}
for some \({ \xi_\alpha \in \partial \mathcal{R}(u_\alpha) }\) and \({ \xi_{ \alpha, k} \in \partial \mathcal{R}(u_{ \alpha, k}) }\).
If we now test with \({u_{ \alpha, k} - u_\alpha}\) in both of the equations in~\eqref{eq:regularized-equations} and subtract the resulting expressions from each other,
we find by monotonicity of \({\partial \mathcal{R}}\) and \({\mathcal{A}}\) the basic stability estimate
\begin{align} \label{eq:basic-stability-estimate}
\begin{aligned}
0 &\leq \dualitypairing[\big]{ \xi_{ \alpha, k} -  \xi_\alpha}{u_{\alpha, k} - u_\alpha} \\[1ex]
&\leq \dualitypairing[\big]{\xi_{ \alpha, k} - \xi_\alpha + \tfrac{ 1 }{ \alpha } \bigl(  \mathcal{A}u_{ \alpha, k} - \mathcal{A}u_\alpha \bigr)}{u_{ \alpha, k} - u_\alpha} \\[1ex]
&= \frac{ 1 }{ \alpha } \dualitypairing[\big]{ f_k - f }{ u_{ \alpha, k} - u_\alpha } \\[1ex]
&\leq \frac{1}{\alpha} \norm{ f_k - f }_{ \X^* } \norm{u_{ \alpha, k} - u_\alpha}.
\end{aligned}
\end{align}
Since \({ \norm{ u_{ \alpha, k} - u_\alpha }  }\), as a \({ k }\)-indexed sequence, is bounded by \cref{thm:boundedness},
it then follows that 
\begin{equation}  \label{eq:stability-convergence}
\lim_{ k \to + \infty }\dualitypairing[\big]{ \xi_{ \alpha, k} -  \xi_\alpha}{u_{\alpha, k} - u_\alpha} = 0.
\end{equation}
Moreover, by reflexivity of \({ \X }\) and boundedness, \({\Set{u_{\alpha, k}}_k}\) converges weakly---up to a subsequence---to some~\({u \in \X }\),
and from~\eqref{eq:stability-convergence} we infer that
\begin{align} \label{eq:stability-limsup-R}
\begin{aligned}
\limsup_{k} \mathcal{R}(u_{\alpha, k}) &\leq \mathcal{R}(u_\alpha) + \lim_{ k } \dualitypairing[\big]{ \xi_{ \alpha, k} }{u_{ \alpha, k} - u_{ \alpha}} \\
&= \mathcal{R}(u_\alpha) + \dualitypairing[\big]{\xi_\alpha}{u - u_\alpha} \\
&< \infty.
\end{aligned}
\end{align}
\Cref{assumptions}~\ref{assumptions:coercivity} now implies that \({u_{\alpha, k} \to u}\) strongly, again up to a subsequence,
and so demicontinuity results in \({\mathcal{A}u_{\alpha, k} \rightharpoonup \mathcal{A}u}\) weakly in~\({ \X^* }\).
In combination with the convergence \({f_k \to f}\) and~\eqref{eq:stability-convergence},
which by virtue of the definitions of \({\xi_{\alpha,k}}\) and \({\xi_\alpha}\) can be rewritten as
\begin{equation*}
\lim_{k} \dualitypairing[\big]{\mathcal{A}u_{\alpha} - \mathcal{A}u_{\alpha, k} + f - f_k}{u_\alpha - u_{\alpha, k}} = 0,
\end{equation*}
we therefore deduce that
\begin{equation*}
\dualitypairing*{\mathcal{A}u_\alpha - \mathcal{A}u}{u_\alpha - u} = 0.
\end{equation*}
Strict monotonicity now forces \({u = u_\alpha}\),
and standard subsequence-subsequence reasoning yields that the full sequence~\({\Set{u_{\alpha, k}}_k}\) converges to~\({u_\alpha}\).
As a consequence, \eqref{eq:stability-limsup-R}~becomes
\begin{equation*}
\limsup_{k} \mathcal{R}(u_{\alpha, k}) \leq \mathcal{R}(u_\alpha),
\end{equation*}
so that \({\mathcal{R}(u_{\alpha, k}) \to \mathcal{R}(u_\alpha)}\) by lower semicontinuity.
Finally, demiconinuity gives \({ \mathcal{A}u_{\alpha, k} \rightharpoonup \mathcal{A}u_\alpha }\) weakly,
and this concludes the proof of \cref{thm:stability}.

\subsection{Convergence---proof of \cref{thm:convergence}}

Note first that, similarly to the basic stability estimate~\eqref{eq:basic-stability-estimate}, we have
\begin{align} \label{eq:basic-convergence-estimate}
\begin{aligned}
\mathcal{R}(u_\alpha^\delta) - \mathcal{R}(u^\dagger) &\leq \dualitypairing[\big]{\xi_\alpha^\delta}{u_\alpha-u^\dagger} \\[1ex]
&= \dualitypairing[\big]{\xi_\alpha^\delta + \tfrac{1}{\alpha}(\mathcal{A}u_\alpha^\delta-\mathcal{A}u^\dagger)}{u_\alpha^\delta-u^\dagger} - \frac{1}{\alpha}\dualitypairing[\big]{\mathcal{A}u_\alpha^\delta-\mathcal{A}u^\dagger}{u_\alpha^\delta-u^\dagger} \\[1ex]
&= \frac{1}{\alpha}\dualitypairing[\big]{f^\delta-f^\dagger}{u_\alpha^\delta-u^\dagger} - \frac{1}{\alpha}\dualitypairing[\big]{\mathcal{A}u_\alpha^\delta-\mathcal{A}u^\dagger}{u_\alpha^\delta-u^\dagger}\\[1ex]
&\leq \frac{\delta}{\alpha}\norm{u_\alpha^\delta-u^\dagger} - \frac{1}{\alpha}\dualitypairing[\big]{\mathcal{A}u_\alpha^\delta-\mathcal{A}u^\dagger}{u_\alpha^\delta-u^\dagger}.
\end{aligned}
\end{align}
If \({ \Set[\big]{ u_{\alpha}^\delta }{}_\delta^{} }\) is bounded in~\({ \X }\),
where \({ \alpha = \alpha(\delta) }\) is as in~\eqref{eq:convergence-apriori},
then it converges weakly---up to a subsequence---to some \({u \in \X }\) by reflexivity.
From~\eqref{eq:basic-convergence-estimate}, monotonicity of~\({ \mathcal{A} }\), and the hypothesis we furthermore have
\begin{equation}\label{eq:basic-convergence-estimate:2}
\limsup_{\delta \to 0} \mathcal{R}(u_{\alpha}^\delta) \leq \mathcal{R}(u^\dagger) + \limsup_{\delta \to 0} \frac{\delta}{\alpha(\delta)}\norm{u_{\alpha}^\delta - u^\dagger} \leq C < \infty
\end{equation}
for some constant \({ C }\) depending only on \({ \sup_\delta( \norm{ u_\alpha^\delta } ) }\).
By compactness in \Cref{assumptions}~\ref{assumptions:coercivity},
it then follows that \({u_\alpha^\delta \to u}\) as \({ \delta \to 0 }\),
and consequently \({ \mathcal{A}u_\alpha^\delta \rightharpoonup \mathcal{A}u }\) weakly in~\({ \X^* }\) by demicontinuity.
Now we may write~\eqref{eq:basic-convergence-estimate} as
\begin{equation*}
0 \leq \dualitypairing[\big]{\mathcal{A}u_\alpha^\delta-\mathcal{A}u^\dagger}{u_\alpha^\delta-u^\dagger} \leq \delta \norm{ u_\alpha^\delta - u^\dagger } - \alpha \left( \mathcal{R}(u_\alpha^\delta) - \mathcal{R}(u^\dagger) \right)
\end{equation*}
and take limits to obtain
\begin{equation*}
\dualitypairing{ \mathcal{A}u - \mathcal{A}u^\dagger}{u - u^\dagger} = 0.
\end{equation*}
Strict monotonicity gives \({u = u^\dagger}\),
and uniqueness of \({u^\dagger}\) also implies that the full sequence \({\Set[\big]{ u_{\alpha}^\delta }{}_\delta^{}}\) converges to~\({u^\dagger}\).
Thus~\eqref{eq:basic-convergence-estimate:2} in fact implies that
\[
  \limsup_{\delta \to 0} \mathcal{R}(u_\alpha^\delta) \leq \mathcal{R}(u^\dagger),
\]
which, in view of the lower semicontinuity of \({ \mathcal{R} }\) implies that
\({ \mathcal{R}(u_\alpha^\delta) \to \mathcal{R}(u^\dagger) }\).

Accordingly, it remains to establish boundedness of \({\Set[\big]{ u_{\alpha}^\delta }{}_\delta^{}}\).
By assumption, \({\delta / \alpha(\delta)}\) is bounded as~\({\delta \to 0}\),
and thus there exists \({\widetilde{C} > 0}\) such that \({\delta/\alpha(\delta) \leq \widetilde{C}}\)
for all sufficiently small~\({\delta}\).
As such, we can again replace the constant \({C}\) in~\eqref{eq:regulariser-contradiction-bnd} by \({\widetilde{C}}\)
and correspondingly modify the values of \({K}\) and~\({r_0}\).
Then we obtain by following the proof of \cref{thm:existence}
and applying \cref{thm:browder} the estimate
\({\norm{u_\alpha^\delta-u^\dagger} \leq r_0}\) for all sufficiently small~\({\delta}\),
which establishes the boundedness of \({\Set[\big]{ u_{\alpha}^\delta }{}_\delta^{}}\)
and concludes the proof of \cref{thm:convergence}.

\section{Proofs of convergence rates}
\label{sec:convergence-rates}

In the following we will establish \cref{thm:rates-main,thm:rates-norm},
where we restate the general variational source condition
\begin{equation} \label{eq:varineq-restated}
D(u, u^\dagger) \leq \mathcal{R}(u) - \mathcal{R}(u^\dagger) + \varphi\bigl(\dualitypairing{ \mathcal{A}u - \mathcal{A}u^\dagger }{ u - u^\dagger }  \bigr)
\end{equation}
from~\eqref{eq:varineq} for convenience.
We write \({ \mathcal{A}u^\dagger }\) instead of \({ f^\dagger }\) in order to highlight the monotone structure.
Here, \({ D }\) is any nonnegative function on~\({ \X \times \X }\) acting as a distance measure,
and \({ \varphi \colon [0, \infty) \to [0, \infty) }\) denotes an \emph{index function},
meaning that \({ \varphi }\) is concave, strictly increasing, and continuous with \({ \varphi(0) = 0 }\).
We also assume a sublinear rate of decay of \({ \varphi }\) at the origin---that is, \({ \lim_{ t \to 0^+ } \varphi(t) / t = \infty }\).

In addition, let \({ \psi \colon [0, \infty) \to [0, \infty) }\) be the \emph{convex (Fenchel) conjugate} of~\({ \varphi^{-1} }\) defined by
\begin{equation} \label{eq:convex-conjugate}
\psi(s) = \sup_{ t \geq 0 } \left( st - \varphi^{-1}(t) \right).
\end{equation}
By construction, this mapping is convex and, due to the nonnegativity of \({ \varphi }\) plus the fact that \({ \varphi(0) = 0 }\), also nonnegative with \({ \psi(0) = 0 }\).
This further implies that \({ \psi }\) is continuous and increasing on~\({ [0, \infty) }\).
Moreover, since \({ \varphi }\) decays sublinearly at~\({ 0 }\),
the function~\({ \psi }\) is strictly increasing near~\({ 0 }\)---and therefore also on its whole domain by convexity.

\begin{proof}[of~\cref{thm:rates-main}]
With help of the basic stability estimate
\begin{equation*}
\mathcal{R}(u_\alpha^\delta) - \mathcal{R}(u^\dagger) \leq \frac{\delta}{\alpha}\norm{u_\alpha^\delta-u^\dagger} - \frac{1}{\alpha}\dualitypairing[\big]{\mathcal{A}u_\alpha^\delta-\mathcal{A}u^\dagger}{u_\alpha^\delta-u^\dagger}
\end{equation*}
from~\eqref{eq:basic-convergence-estimate},
we find that the general variational inequality~\eqref{eq:varineq-restated} yields
\begin{align*}
D(u_\alpha^\delta, u^\dagger) &\leq \mathcal{R}(u_\alpha^\delta) - \mathcal{R}(u^\dagger) + \varphi\Bigl(\dualitypairing[\big]{ \mathcal{A}u_\alpha^\delta - \mathcal{A}u^\dagger }{ u_\alpha^\delta - u^\dagger } \Bigr)\\[1ex]
& \leq \frac{\delta}{\alpha}\norm{u_\alpha^\delta - u^\dagger} - \frac{1}{\alpha}\dualitypairing[\big]{\mathcal{A}u_\alpha^\delta - \mathcal{A}u^\dagger}{u_\alpha^\delta - u^\dagger} + \varphi\Bigl(\dualitypairing[\big]{ \mathcal{A}u_\alpha^\delta - \mathcal{A}u^\dagger }{ u_\alpha^\delta - u^\dagger }  \Bigr)\\[1ex]
 &\leq \frac{\delta}{\alpha}\norm{u_\alpha^\delta-u^\dagger} + \sup_{t \geq 0} \left(\varphi(t) - \frac{t}{\alpha} \right).
\end{align*}
Since by definition of \({ \psi }\) we have
\begin{equation*}
\sup_{t \geq 0} \left(\varphi(t) - \frac{t}{\alpha} \right) = \frac{1}{\alpha} \sup_{t\geq 0} \, \bigl(\alpha\varphi(t) - t \bigr) = \frac{1}{\alpha} \sup_{s\geq 0} \,\bigl( \alpha s - \varphi^{-1}(s) \bigr)= \frac{\psi(\alpha)}{\alpha},
\end{equation*}
it follows that
\begin{equation}\label{eq:estD}
D(u_\alpha^\delta, u^\dagger) \leq \frac{ \delta }{ \alpha }  \norm{u_\alpha^\delta- u^\dagger} + \frac{\psi(\alpha)}{\alpha}.
\end{equation}
Now recall that the regularisation method~\eqref{eq:subgradient-lavrentiev-reg} is convergent by \cref{thm:convergence}---in fact, weak convergence together with reflexivity of~\({ \X }\) suffice here.
Therefore, for sufficiently a small noise level~\({ \delta }\) and an appropriate parameter choice for~\({ \alpha }\),
we know that \({ u_\alpha^\delta - u^\dagger }\) is bounded in~\({ \X }\) and \({ u_\alpha^\delta \in \mathcal{M} }\).
Hence, there exists \({ C > 0 }\) such that \({ \norm{u_\alpha^\delta - u^\dagger} \leq C }\),
and we obtain the estimate claimed in the statement of the theorem.
At last, choosing \({\alpha \sim \psi^{-1}(\delta) }\) in~\eqref{eq:estD} gives the convergence rate
\begin{equation*}
D(u_\alpha^\delta, u^\dagger) \lesssim \frac{\delta}{\psi^{-1}(\delta)}.
\end{equation*}
\end{proof}

\enlargethispage{\baselineskip} 

\begin{corollary} \label{thm:rates-holder}
Assume that a variational source condition of Hölder type
\begin{equation*}
D(u, u^\dagger) \leq \mathcal{R}(u) - \mathcal{R}(u^\dagger) + c \left(\dualitypairing{ \mathcal{A}u - \mathcal{A}u^\dagger }{ u - u^\dagger }\right)^{1/p}
\end{equation*}
holds for some \({ p > 1 }\) and~\({ c > 0 }\).
Then as~\({ \delta \to 0 }\) we obtain the convergence rate
\begin{equation*}
D(u_\alpha^\delta, u^\dagger) \lesssim \delta^{1/p} \qquad \text{if } \alpha \sim \delta^{1 / q},
\end{equation*}
where \({ q \coloneqq p / (p - 1) > 1 }\) is the conjugate Hölder exponent of~\({ p }\).
\end{corollary}
\begin{proof}
This is immediate from \cref{thm:rates-main},
noting that for \({ \varphi \sim \delta^{1/p} }\) we have \({ \psi \sim \delta^{q} }\),
which leads to \({ \psi^{-1} \sim \delta^{1/ q} }\) and \({ \delta / \psi^{-1}(\delta) \sim \delta^{1 / p} }\).
\end{proof}

Finally we specialise \cref{thm:rates-holder} to norm-like distance measures of the form \({ D(u, u^\dagger) \sim \norm{ u - u^\dagger }^{ r }  }\) for some \({ r > 1 }\).
This results in \cref{thm:rates-norm} by the following arguments.

\begin{proof}[of~\cref{thm:rates-norm}]
We proceed along the proof of \cref{thm:rates-main} until estimate~\eqref{eq:estD},
which in this case reads
\begin{equation*}
c_1 \norm{u_\alpha^\delta - u^\dagger}^r \leq \frac{\delta}{ \alpha } \norm{u_\alpha^\delta - u^\dagger} + c_3 \alpha^{q - 1}
\end{equation*}
for some constant \({ c_ 3 > 0 }\) and with conjugate Hölder exponent \({ q \coloneqq p / (p - 1) }\).
Next we apply Young's inequality \({ ab \leq \frac{ 1 }{ r }a^r + \frac{ 1 }{ s }b^s }\) with conjugate exponent \({ s \coloneqq r / (r - 1) }\) to \({ a \coloneqq  \norm{u_\alpha^\delta - u^\dagger} }\) and \({b \coloneqq 2\delta/(rc_1\alpha)}\) and obtain that
\begin{alignat*}{2} \SwapAboveDisplaySkip
c_1\norm{u_\alpha^\delta - u^\dagger}^r &\leq \frac{rc_1}{2} \frac{2\delta \norm{u_\alpha^\delta - u^\dagger}}{r c_1 \alpha} &&+ c_3 \alpha^{q - 1} \\
&\leq \frac{c_1}{2} \norm{u_\alpha^\delta - u^\dagger}^r + \frac{rc_1}{2s}\left(\frac{2\delta}{rc_1\alpha}\right)^s &&+ c_3 \alpha^{q - 1}.
\end{alignat*}
This simplifies to
\begin{equation*}
\norm{u_\alpha^\delta - u^\dagger}^r \lesssim \frac{ \delta^s }{ \alpha^s } + \alpha^{q - 1},
\end{equation*}
and we may now choose
\begin{equation*}
\alpha \sim \delta \tothepowerof \left( \frac{ s }{ s + q - 1 } \right) =  \delta \tothepowerof \left(\frac{r(p - 1)}{rp - 1}\right),
\end{equation*}
remembering the definitions of \({ q }\) and \({ s }\)---in particular, that \({ (p - 1)(q - 1) = 1 }\)---to arrive at the quantitative error estimate
\begin{equation*}
\norm{u_\alpha^\delta - u^\dagger} \lesssim \delta \tothepowerof \left( \frac{ 1 }{ rp - 1 } \right).
\end{equation*}
\end{proof}

\begin{remark} \label{rem:rates}
Suppose that \({ \mathcal{A} }\) is positively homogeneous with \({ \mathcal{A}(0) = 0 }\)---for instance, that \({ \mathcal{A} }\) is linear---and assume that \({ u^\dagger \neq 0 }\).
In particular, \({ \dualitypairing{ \mathcal{A}u^\dagger }{ u^\dagger } > 0 }\) by strict monotonicity.
If we test with \({ u = (1 - \epsilon) u^\dagger }\) for \({ \epsilon \in (0, 1) }\) in the general variational source condition~\eqref{eq:varineq},
then convexity of \({ \mathcal{R} }\) yields that
\begin{equation} \label{eq:rates-realistic}
D \bigl( (1 - \epsilon)u^\dagger, u^\dagger \bigr) \leq \epsilon \left( \mathcal{R}(0) - \mathcal{R}(u^\dagger) \right) + \varphi \bigl( \epsilon^2 \dualitypairing{ \mathcal{A}u^\dagger }{ u^\dagger }  \bigr).
\end{equation}
By requiring \({ \mathcal{R}(u^\dagger) > \mathcal{R}(0) }\) and finiteness of
\begin{equation*}
\lim_{ \epsilon \to 0^+ } \frac{ 1 }{ \epsilon } D \bigl( (1 - \epsilon)u^\dagger, u^\dagger \bigr),
\end{equation*}
which is automatic if \({ D(u, v) \sim \norm{ u - v }^r }\) for some~\({ r \geq 1 }\),
we may now divide by \({ \sqrt{t}  \coloneqq \epsilon \sqrt{ \dualitypairing{ \mathcal{A}u^\dagger }{ u^\dagger } }}\) in~\eqref{eq:rates-realistic} and let~\({ \epsilon \to 0^+ }\).
This gives
\begin{equation*}
\varphi(t) \gtrsim \sqrt{t}
\end{equation*}
uniformly near~\({ 0 }\),
and hence, we infer that the variational source condition~\eqref{eq:varineq} is realistic only for~\({ \varphi(t) }\) decaying no faster than~\({ \sqrt{t} }\) as~\({ t \to 0^+ }\).
In particular, this corresponds to \({ p \geq 2 }\) in \cref{thm:rates-norm,thm:rates-holder} and agrees with the discussion in \autocite[Section~3.2]{HofKalRes2016a}.
\end{remark}

\section{TV denoising in convolutional Volterra integral problems}
\label{sec:volterra}

We now discuss an application of subgradient-based Lavrentiev regularisation~\eqref{eq:subgradient-lavrentiev-reg} for TV denoising in a class of Volterra integral operators~\eqref{eq:volterra-operator} of the first kind,
where we focus on the linear, convolutional case
\begin{equation} \label{eq:op-Volterra-convolution}
\mathcal{A}u(t) \coloneqq \int_{ 0 }^{ t } k(t - \tau) \, u(\tau) \dee \tau,
\end{equation}
with \({ \mathcal{A} }\) acting from the reflexive space \({ \X \coloneqq \textnormal{L}^p(0, T) }\) to \({ \X^* = \textnormal{L}^q(0, T) }\) for suitable kernels~\({ k }\) and conjugate exponents \({ p, q \in (1, \infty) }\).
Remember that \({ \mathcal{A} }\) retains causality with respect to time~\({ t }\),
that is, \({ \mathcal{A}u(t) }\) depends only on \({ u(\tau) }\) for~\({ \tau \leq t }\).
Note nevertheless that the theory may also be applied to more global, noncausal problems,
involving, for instance, Fredholm integral operators.
One may also extend the results to operators \({ u \mapsto \mathcal{A}u + a(\cdot, u(\cdot))  }\) with pointwise terms added,
subject to suitable growth conditions.

In the context of TV denoising,
we next define the regularisation functional \({ \mathcal{R} \colon\X \to [0, \infty] }\) as
\begin{equation} \label{eq:total-variation-regulariser}
\mathcal{R}(u) \coloneqq
\begin{cases}
\seminorm{u}_{ \textnormal{TV}} & \text{ if } u \in \textnormal{BV}(0, T); \\
+\infty &\text{ else,}
\end{cases}
\end{equation}
where \({ \textnormal{BV}(0, T) }\) is the space of integrable functions with bounded variation and seminorm
\begin{equation*}
\seminorm{u}_{ \textnormal{TV}} \coloneqq \abs{ \textnormal{D}u }(0, T) \coloneqq \sup \Set*{ -\int_{ 0 }^{ T } \varphi' u \dee t \given \varphi \in \textnormal{C}_{ \textnormal{c}}^1(0, T) \text{ with } \norm{ \varphi }_{ \textnormal{L}^\infty(0, T) } \leq 1 }.
\end{equation*}
Then \({ \mathcal{R} }\) satisfies \Cref{assumptions}~\ref{assumptions:regulariser},
and since \({ \textnormal{BV}(0, T) }\) is compactly embedded in \({ \X }\) for every \({ p \in (1, \infty) }\) (see~\autocite[Proposition~3.13]{AmbFusPal2000a}),
it follows that the sublevel sets in \Cref{assumptions}~\ref{assumptions:coercivity} are compact as well.
Note also that the growth condition~\eqref{eq:coercivity-growth} holds automatically because \({ \mathcal{A} }\) is linear.
Accordingly, it remains to establish strict monotonicity of~\({ \mathcal{A} }\) on~\({ \X }\),
which we are able to prove for a class of convolution kernels and exponents~\({ p }\).

\subsection{Strictly monotone Volterra operators}

If \({ k }\) is identically~\({ 1 }\), the inverse problem for \({ \mathcal{A}u = f^\dagger }\) becomes that of numerical differentiation.
In this case, the monotonicity of \({ \mathcal{A} }\) on \({ \textnormal{L}^2(0, T) }\), say, follows directly from
\begin{equation*}
\dualitypairing{ \mathcal{A}u }{ u } = \int_{ 0 }^{ T } \mathcal{A}u(t) \, u(t) \dee t =  \frac{ 1 }{ 2 } \left( \int_{ 0 }^{ T } u(t) \dee t \right)^2
\end{equation*}
by integrating by parts,
which implies that \({ \mathcal{A} }\) is strictly monotone on the quotient Hilbert space
\begin{equation*}
\textnormal{L}^2(0, T) / \Set[\big]{ u \in \textnormal{L}^2(0, T) \given \textstyle\int_{ 0 }^{ T }  u \dee \tau = 0 }.
\end{equation*}
More generally, we have the fractional counterpart with singular convolution kernel
\begin{equation} \label{eq:abel-kernel}
k(x) = \frac{ x^{s - 1} }{ \Upgamma(s) }
\end{equation}
for \({ s > 0 }\), where \({ \Upgamma }\) is the Gamma function,
in which \eqref{eq:op-Volterra-convolution} is known as an Abel or Riemann--Liouville integral.
The latter kernel and also the negative exponential \({ x \mapsto \exp (- cx) }\)
both generate strictly monotone operators satisfying \Cref{assumptions}~\ref{assumptions:operator},
as seen in the following result.

\begin{proposition} \label{thm:volterra-strictly-monotone}
Let \({ p \in (1, 2] }\) and \({ q }\) be its conjugate exponent.
Assume that the kernel~\({ k }\) is strictly convex and decreasing on~\({ (0, T] }\) and lies in~\({ \textnormal{L}^{q/2}(0, T) }\),
with strictly positive mean, that is, \({ \int_{ 0 }^{ T } k(t) \dee t > 0 }\).
Then the convolutional Volterra operator~\eqref{eq:op-Volterra-convolution} defines a strictly monotone, bounded map
\begin{equation*}
\mathcal{A} \colon \textnormal{L}^p(0,T) \to \textnormal{L}^q(0, T).
\end{equation*}
\end{proposition}
\begin{remark}
Note that we both allow singularities in~\({ k }\) at the origin and some negative values as long as the mean of \({ k }\) is strictly positive.
The proof below is a simple adaption of~\autocite[Lemma~1.2]{Ask2015a}
and is, in some sense, related to Bochner's result characterising the Fourier transform of positive-definite functions.
We refer to~\autocite{NohShe1976a} and \autocite{Sta1976a} for further results in this direction, the latter of which includes monotonicity of Volterra operators \({ \int_{ 0 }^{ t } u(\tau) \dee \mu(\tau) }\) with respect to a class of positive-definite Radon measures~\({ \mu }\).
It is, however, not entirely clear how these results relate to strict monotonicity.
\end{remark}
\begin{proof} 
Hölder's inequality yields that
\begin{equation*}
\abs{ \mathcal{A}u(t) } \leq \norm{ k }_{ \textnormal{L}^{q/2}(0, T) }^{ 1/2 } \left( \int_{ 0 }^{ t } \abs{ k(t - \tau) }^{p/2} \abs{u(\tau) }^p \dee \tau \right)^{1/p} \!\!\!\!\!\!\!\!\!,
\end{equation*}
and so the general Minkowski integral inequality leads to
\begin{align*}
\norm{ \mathcal{A}u }_{ \textnormal{L}^q(0, T) }^{} &\leq \norm{ k }_{ \textnormal{L}^{q/2}(0, T) }^{ 1/2 } \Biggl( \int_{ 0 }^{ T } \left( \int_{ 0 }^{ t } \abs{ k(t - \tau) }^{p/2} \abs{u(\tau) }^p \dee \tau \right)^{q/p} \!\!\!\!\!\!\! \dee t \Biggr)^{1/q} \\[1ex]
& \leq \norm{ k }_{ \textnormal{L}^{q/2}(0, T) }^{ 1/2 } \Biggl( \int_{ 0 }^{ T } \left( \int_{ \tau }^{ T } \abs{ k(t - \tau) }^{q/2} \abs{u(\tau) }^q \dee t \right)^{p/q} \!\!\!\!\!\!\! \dee \tau \Biggr)^{1/p} \\[1ex]
& \leq \norm{ k }_{ \textnormal{L}^{q/2}(0, T) } \norm{ u }_{ \textnormal{L}^p(0, T) },
\end{align*}
which establishes boundedness of \({ \mathcal{A} \colon \textnormal{L}^p(0,T) \to \textnormal{L}^q(0, T) }\).

As regards strict monotonicity, we first approximate \({ k }\) pointwise by the shifted, bounded functions
\begin{equation*}
k_\epsilon(t) \coloneqq \left\{
\begin{aligned}
&k(t + \epsilon) && \text{for } t \in [0, T - \epsilon]; &\\
&k(T) && \text{for } t \in (T - \epsilon, T] &
\end{aligned}
\right.
\end{equation*}
in order to avoid the potential singularity at~\({ 0 }\).
Then after extending \({ k_\epsilon }\) evenly to~\({ [-T, 0] }\), Dirichlet's theorem gives that \({ k_\epsilon }\) (which is integrable, continuous, and of bounded variation) equals its Fourier series
\begin{equation*}
k_\epsilon(t) = \frac{c_{0, \epsilon}}{2} + \sum_{ n = 1 }^{ \infty } c_{n, \epsilon} \cos( \widehat{ n } t)
\end{equation*}
pointwise on \({ (-T, T) }\),
where \({ c_{n, \epsilon} \coloneqq \frac{ 2 }{ T } \int_{ 0 }^{ T } k_\epsilon(t) \cos( \widehat{ n } t)  \dee t }\) and \({ \widehat{ n } \coloneqq 2 \uppi n / T  }\).
Since \({ k }\) is strictly convex and decreasing on~\({ (0, T] }\),
it follows from \autocite[Lemma~1.1]{Ask2015a} that
\begin{equation*}
c_n \coloneqq \lim_{ \epsilon \to 0^+ } c_{n, \epsilon} = \frac{ 2 }{ T } \int_{ 0 }^{ T } k(t) \cos( \widehat{ n } t)  \dee t > 0
\end{equation*}
for all~\({ n \geq 1 }\), and \({ c_0 \coloneqq \lim\limits_{ \epsilon \to 0^+ } c_{0, \epsilon} = \frac{ 2 }{ T } \int_{ 0 }^{ T } k(t) \dee t > 0 }\) by assumption.

We next use dominated convergence repeatedly to find that
\begin{align*}
\dualitypairing{ \mathcal{A}u }{ u } &= \lim_{ \epsilon \to 0^+ } \int_{ 0 }^{ T } \int_{ 0 }^{ t } k_\epsilon(t - \tau) \, u(\tau) \dee \tau \, u(t) \dee t \\[1ex]
&= \int_{ 0 }^{ T } \left( \frac{ c_0 }{ 2 } \int_{ 0 }^t u(\tau) \dee \tau + \sum_{ n = 1 }^{ \infty } c_n \int_{ 0 }^{ t } u(\tau) \cos \bigl( \widehat{ n } (t - \tau) \bigr)  \dee \tau \right) u(t) \dee t.
\end{align*}
Observe now that
\begin{equation*}
\int_{ 0 }^{ T } \int_{ 0 }^{ t } u(\tau) \cos \bigl( \widehat{ n } (t - \tau) \bigr)  \dee \tau \, u(t) \dee t = \int_{ 0 }^{ T } \int_{ 0 }^{ t } v(\tau) \dee \tau \, v(t) \dee t + \int_{ 0 }^{ T } \int_{ 0 }^{ t } w(\tau) \dee \tau \, w(t) \dee t
\end{equation*}
with \({ v(t) \coloneqq u(t) \cos( \widehat{ n } t)  }\) and \({ w(t) \coloneqq u(t) \sin( \widehat{ n } t)  }\).
As such, integration by parts gives that
\begin{align*}
\int_{ 0 }^{ T } \int_{ 0 }^{ t } u(\tau) \cos \bigl( \widehat{ n } (t - \tau) \bigr) \dee \tau \, u(t) \dee t &= \frac{ 1 }{ 2 } \left( \int_{ 0 }^{ T } v(t) \dee t \right)^2 + \frac{ 1 }{ 2 }\left( \int_{ 0 }^{ T } w(t) \dee t \right)^2 \\[1ex]
&= \frac{ 1 }{ 2 } \abs[\Bigg]{ \int_{ 0 }^{ T } u(t) \, \textnormal{e}^{ \textnormal{i} \widehat{ n } t } \dee t }^2,
\end{align*}
and so
\begin{equation*}
\dualitypairing{ \mathcal{A}u }{ u } = \frac{ c_0 }{ 2 } \left( \int_{ 0 }^{ T } u(t) \dee t \right)^2 + \sum_{ n = 1 }^{ \infty } \frac{ c_n }{ 2 } \abs[\Bigg]{ \int_{ 0 }^{ T } u(t) \, \textnormal{e}^{ \textnormal{i} \widehat{ n } t } \dee t }^2
= \sum_{ n = 0 }^{ \infty } \frac{ c_n }{ 2 } \abs[\Bigg]{ \int_{ 0 }^{ T } u(t) \, \textnormal{e}^{ \textnormal{i} \widehat{ n } t } \dee t }^2.
\end{equation*}
Since \({ c_n > 0 }\) for all~\({ n \geq 0 }\),
we conclude that \({ \mathcal{A} }\) is strictly monotone.
\end{proof}

\subsection{Description of the subgradient of TV and numerical examples}

In order to make the regularisation method~\eqref{eq:subgradient-lavrentiev-reg} practical numerically,
let us first characterise the subgradient \({ \partial \mathcal{R} \colon \X \rightrightarrows \X^* }\) of~\eqref{eq:total-variation-regulariser}.
Denote by \({ (\textnormal{D}u)^+ }\) and \({ (\textnormal{D}u)^- }\) the positive and negative parts in the Jordan measure decomposition
\begin{equation*}
\textnormal{D}u = (\textnormal{D}u)^+ - (\textnormal{D}u)^-,
\end{equation*}
and let \({ \support \mu }\) symbolise the support of a measure~\({ \mu }\).
We then obtain the following description of~\({ \partial \mathcal{R} }\),
which is a straightforward adaption of~\autocite[Proposition~2.1]{KieMucRyb2013a}.
Note that \({ \textnormal{W}^{1, q}(0, T) }\) means the Sobolev space of functions in \({ \textnormal{L}^q(0, T) }\) whose weak derivative also lie in~\({ \textnormal{L}^q(0, T) }\).

\begin{proposition} \label{thm:tv-subgradient-char}
Let \({ u \in \domain \mathcal{R} = \textnormal{BV}(0, T) \subset \textnormal{L}^p(0, T) }\) for \({ \mathcal{R} }\) in~\eqref{eq:total-variation-regulariser} and \({ p, q \in (1, \infty) }\) be conjugate exponents.
Then \({ \xi \in  \partial \mathcal{R}(u) }\) if and only if \({ \xi = - \varphi' }\) for \({ \varphi \in \textnormal{W}^{1, q}(0, T) }\) vanishing at its endpoints (thus \({ \xi }\) has zero mean) and satisfying
\begin{equation*}
\norm{ \varphi }_{ \textnormal{L}^{\infty }(0, T)} \leq 1, \qquad \support (\textnormal{D}u)^+ \subseteq \Set{ \varphi = 1 }, \qquad \text{and} \qquad \support (\textnormal{D}u)^- \subseteq \Set{ \varphi = -1}. 
\end{equation*}
Moreover, \({ \mathcal{R}(u) = \dualitypairing{ \xi }{ u }  }\).
\end{proposition}

\begin{proof}
As in~\autocite[Proposition~2.1]{KieMucRyb2013a},
we exploit from the Fenchel--Young inequality that
\begin{equation*}
\xi \in \partial \mathcal{R}(u) \quad \Leftrightarrow \quad \dualitypairing{ \xi }{ u } = \mathcal{R}(u) + \mathcal{R}^*( \xi ),
\end{equation*}
where
\begin{equation*}
\mathcal{R}^*( \xi ) \coloneqq \smashoperator{\sup_{ v \in \textnormal{L}^p(0, T) }} \; \bigl(\dualitypairing{ \xi }{ v } - \mathcal{R}(v) \bigr)
\end{equation*}
denotes the convex conjugate of~\({ \mathcal{R} }\) similarly as~\eqref{eq:convex-conjugate}.
Since \({ \mathcal{R}(u) < \infty }\) by assumption, it remains to characterise \({ \domain \mathcal{R}^* }\).

It suffices to let the supremum run over \({ v \in \textnormal{BV}(0, T) }\) and we need only consider \({ \xi \in \textnormal{L}^q(0, T) }\) with zero mean---if not, we get \({ \mathcal{R}^*( \xi ) = \infty }\) by letting \({ v }\) be constant.
It is also clear that \({ \xi }\) corresponds uniquely to a function \({ \varphi \in \textnormal{W}^{1, q}(0, T) }\) vanishing at its endpoints and given by
\begin{equation*}
\varphi(t) \coloneqq - \int_{ 0 }^{ t } \xi(\tau) \dee \tau.
\end{equation*}
By integrating by parts, we obtain that
\begin{equation*}
\int_{ 0 }^{ T } \varphi \dee \,(\textnormal{D}v) + \int_{ 0 }^{ T } \varphi' v \dee t = \varphi(T) v(T^-) - \varphi(0) v(0^+) = 0,
\end{equation*}
and so
\begin{align} \SwapAboveDisplaySkip
\dualitypairing{ \xi }{ v } - \seminorm{v}_{ \textnormal{TV}} &= - \int_{ 0 }^{ T } \varphi' v \dee t - \int_{ 0 }^{ T } \dee \, \abs{ \textnormal{D}v } \notag \\[1ex]
&= \int_{ 0 }^{ T } (\varphi - 1) \dee \,(\textnormal{D}v)^+ - \int_{ 0 }^{ T } (\varphi + 1) \dee \,(\textnormal{D}v)^-. \label{eq:tv-conjugate-calc}
\end{align}
If \({ \varphi \geq -1 }\) and \({ \Set{ \varphi > 1 } }\) has positive measure (or \({ \varphi \leq 1 }\) and \({ \Set{ \varphi < -1 } }\) has positive measure),
then one finds that \({ \mathcal{R}^*( \xi ) = \infty }\) by choosing, say, a sequence of monotonically increasing (decreasing) functions~\({ v \in \textnormal{BV}(0, T) }\).
As such, we must have \({ \norm{ \varphi }_{ \textnormal{L}^\infty(0, T) } \leq 1 }\).
But then~\eqref{eq:tv-conjugate-calc} is nonpositive---and equals~\({ 0 }\), corresponding to the supremum over~\({ v }\), precisely when
\begin{equation*}
(\textnormal{D}v)^{+}(\Set{ \varphi < 1 }) = 0 \qquad \text{and} \qquad (\textnormal{D}v)^{-}(\Set{ \varphi > -1 }) = 0.
\end{equation*}

In conclusion, \({\mathcal{R}^*( \xi ) = 0 }\) for \({ \xi \in \domain \mathcal{R}^* }\),
so that \({ \mathcal{R}(u) = \dualitypairing{ \xi }{ u }  }\) and the supremum in the definition of \({ \mathcal{R}^* }\) is achieved at~\({ v = u }\).
\end{proof}

By writing the regularised problem~\eqref{eq:subgradient-lavrentiev-reg} as \({ \mathcal{A}u - f^\delta \in - \alpha \partial \mathcal{R}(u) }\),
with \({ u \coloneqq u_\alpha^\delta }\) for brevity,
and introducing
\begin{equation*}
\mathcal{L}u(t) \coloneqq \int_{ 0 }^{ t } \left( \mathcal{A}u(\tau) - f^\delta(\tau) \right) \dee \tau,
\end{equation*}
the characterisation of~\({ \partial \mathcal{R} }\) in \cref{thm:tv-subgradient-char} now leads after integration to the following relations for \({ \abs{ \textnormal{D}u } }\)-almost every \({ t \in (0, T) }\):
\begin{empheq}[left = \empheqlbrace]{alignat=3}
\; &\abs{\mathcal{L}u(t) } &&\leq \alpha; &\label{eq:optimality-conds-tube} \\[1ex]
&\hphantom{|}\mathcal{L}u(t) &&= + \alpha  \qquad \text{when } t \in \support (\textnormal{D}u)^+; &\label{eq:optimality-conds-pos} \\[1ex]
&\hphantom{|}\mathcal{L}u(t) &&= - \alpha  \qquad \text{when } t \in \support (\textnormal{D}u)^-; &\label{eq:optimality-conds-neg} \\[1ex]
&\hphantom{|}\mathcal{L}u(T) &&= 0. & \label{eq:optimality-conds-sum}
\end{empheq}
We see from~\eqref{eq:optimality-conds-tube} that the integrated image \({ \int_{ 0 }^{ t } \mathcal{A}u \dee \tau }\) of the approximate solution lies within an \({ \alpha }\)-tube of the integrated data~\({ \int_{ 0 }^{ t } f^\delta \dee \tau }\).
Moreover, touching the upper part of the tube~\eqref{eq:optimality-conds-pos} may lead to an increase in~\({ u }\),
whereas \({ u }\) has a chance to decrease whenever one touches the lower part of the tube~\eqref{eq:optimality-conds-neg}.
Then it is possible to compute a piecewise-constant solution in \({\textnormal{BV}(0, T) }\) of the system~\eqref{eq:optimality-conds-tube}--\eqref{eq:optimality-conds-sum} by means of a new version of the classical taut-string algorithm~\autocite{DavKov2001a,Gra2007a,Con2013a},
which we will discuss in detail in a forthcoming paper.
The approach is very fast, efficient and employs dynamic programming.
In addition to \({ \mathcal{A} }\) being strictly monotone in the functional-analytic sense,
we also require in this context that \({ \mathcal{A} }\) has a positive kernel~\({ k }\)---so that \({ \mathcal{A} }\) preserves the pointwise order.
That is, if \({ u, v \in \textnormal{BV}(0,T) }\) satisfy
\({ u(t) \geq v(t) }\) for almost every \({ t \in (0,T) }\),
then also \({ \mathcal{A}u(t) \geq \mathcal{A}v(t) }\) for almost every~\({ t }\).

As illustrations for TV denoising with~\eqref{eq:subgradient-lavrentiev-reg} in convolutional Volterra equations~\eqref{eq:op-Volterra-convolution}, we first consider the Abel operator with kernel~\eqref{eq:abel-kernel} for \({ s = 1/3 }\).
\Cref{fig:volterra-abel} displays the typical performance of the approximate solutions~\({ u_\alpha^\delta }\) with moderate noise in the data,
and we observe quite faithful reconstructions of the original jump locations and heights.

\begin{figure}[h!]
\captionsetup{format=hang}
\vspace*{0em}%
\centering%
\raisebox{4em}[0pt][0pt]{\begin{subfigure}[c]{.2\textwidth}%
\caption{Solutions:\\
\({ u^\dagger }\) (dashed black); \\
\({ u_\alpha^\delta }\) (solid red).}
\end{subfigure}}%
\quad%
\centering%
\setlength\figureheight{.2\linewidth}%
\setlength\figurewidth{.75\textwidth}%
%
%
\colorlet{mycolor1}{WildStrawberry}%
\colorlet{mycolor2}{artcolortext}%
\begin{tikzpicture}

\begin{axis}[%
width=\figurewidth,
height=\figureheight,
at={(0\figurewidth,0\figureheight)},
scale only axis,
hide axis,
xmin=-0.025,
xmax=1.025,
ymin=-17,
ymax=10.1,
]
\draw [arrows = {|-Stealth}] (0,-16) -- (1,-16) node [midway, fill=artcolorbg] {\footnotesize Time \({ t }\)};
\addplot [color=mycolor2,dashed,line width=2.0pt,forget plot]
  table[row sep=crcr]{%
0	-8\\
0.029029029029029	-8\\
0.03003003003003	7\\
0.0990990990990991	7\\
0.1001001001001	-5\\
0.149149149149149	-5\\
0.15015015015015	-13\\
0.199199199199199	-13\\
0.2002002002002	-1\\
0.249249249249249	-1\\
0.25025025025025	3\\
0.309309309309309	3\\
0.31031031031031	0\\
0.379379379379379	0\\
0.38038038038038	-6\\
0.449449449449449	-6\\
0.45045045045045	3\\
0.519519519519519	3\\
0.520520520520521	8\\
0.62962962962963	8\\
0.630630630630631	1\\
0.77977977977978	1\\
0.780780780780781	7\\
0.8998998998999	7\\
0.900900900900901	-3\\
1	-3\\
};

\addplot [color=mycolor1,solid,line width=1.5pt,forget plot]
  table[row sep=crcr]{%
0	-7.66092233612626\\
0.029029029029029	-7.66092233612626\\
0.03003003003003	3.42497544479183\\
0.032032032032032	3.42497544479183\\
0.033033033033033	5.41625836293544\\
0.036036036036036	5.41625836293544\\
0.037037037037037	6.71166106155351\\
0.0820820820820821	6.71166106155351\\
0.0830830830830831	6.39359809173781\\
0.0990990990990991	6.39359809173781\\
0.1001001001001	1.37564853019105\\
0.101101101101101	1.37564853019105\\
0.102102102102102	-4.44796654431455\\
0.119119119119119	-4.44796654431455\\
0.12012012012012	-5.12777658008614\\
0.146146146146146	-5.12777658008614\\
0.147147147147147	-5.74007525951299\\
0.152152152152152	-5.74007525951299\\
0.153153153153153	-12.5362788239431\\
0.2002002002002	-12.5362788239431\\
0.201201201201201	-1.08444552448558\\
0.211211211211211	-1.08444552448558\\
0.212212212212212	-0.506988989989414\\
0.253253253253253	-0.506988989989414\\
0.254254254254254	2.79312141185089\\
0.315315315315315	2.79312141185089\\
0.316316316316316	-0.0687269461230024\\
0.369369369369369	-0.0687269461230024\\
0.37037037037037	-0.126944876549684\\
0.382382382382382	-0.126944876549684\\
0.383383383383383	-5.59198535649894\\
0.448448448448448	-5.59198535649894\\
0.449449449449449	-3.99215517370783\\
0.451451451451451	-3.99215517370783\\
0.452452452452452	0.76815537890766\\
0.459459459459459	0.76815537890766\\
0.46046046046046	3.23528036778794\\
0.523523523523523	3.23528036778794\\
0.524524524524524	7.58368711003963\\
0.53953953953954	7.58368711003963\\
0.540540540540541	7.81471952908762\\
0.621621621621622	7.81471952908762\\
0.622622622622623	6.81657346392052\\
0.623623623623624	6.81657346392052\\
0.624624624624625	3.44780870815387\\
0.631631631631632	3.44780870815387\\
0.632632632632633	2.36232718500623\\
0.648648648648649	2.36232718500623\\
0.64964964964965	1.30145517601352\\
0.678678678678679	1.30145517601352\\
0.67967967967968	1.21203845359852\\
0.770770770770771	1.21203845359852\\
0.771771771771772	1.26983635990426\\
0.776776776776777	1.26983635990426\\
0.777777777777778	1.88800026778653\\
0.781781781781782	1.88800026778653\\
0.782782782782783	6.17426928064825\\
0.7997997997998	6.17426928064825\\
0.800800800800801	6.70481795829099\\
0.835835835835836	6.70481795829099\\
0.836836836836837	6.8494112434135\\
0.900900900900901	6.8494112434135\\
0.901901901901902	1.22449229510369\\
0.904904904904905	1.22449229510369\\
0.905905905905906	-3.07388395228273\\
1	-3.07388395228273\\
};
\end{axis}
\end{tikzpicture}%
%

\vspace*{.5em}
\centering
\raisebox{4em}[0pt][0pt]{\ \begin{subfigure}[c]{.2\textwidth}%
\caption{Data:\\
\({ f^\dagger }\) (dashed black);\\
\({ f^\delta }\) (solid blue).}
\end{subfigure}}%
\quad%
\centering%
\setlength\figureheight{.25\linewidth}%
\setlength\figurewidth{.75\textwidth}%
\input{tikz/volterra_abel_data.tikz}
\caption{TV~denoising for the convolutional Abel operator~\eqref{eq:op-Volterra-convolution}--\eqref{eq:abel-kernel} with \({ s = 1/3 }\).}\label{fig:volterra-abel}
\end{figure}

In the second example, we employ the exponential kernel~\({ x \mapsto \exp(- x / 10) }\) and consider substantially noisy data,
as seen in \cref{fig:volterra-exp}.
Again we find the performance to be pretty good,
although there are some deviations from the true jump heights for jumps occurring close to another one.

\begin{figure}[h!]
\captionsetup{format=hang}
\vspace*{0em}%
\centering%
\raisebox{4em}[0pt][0pt]{\begin{subfigure}[c]{.2\textwidth}%
\caption{Solutions:\\
\({ u^\dagger }\) (dashed black); \\
\({ u_\alpha^\delta }\) (solid red).}
\end{subfigure}}%
\quad%
\centering%
\setlength\figureheight{.2\linewidth}%
\setlength\figurewidth{.75\textwidth}%
%
%
\colorlet{mycolor1}{WildStrawberry}%
\colorlet{mycolor2}{artcolortext}%
\begin{tikzpicture}

\begin{axis}[%
width=\figurewidth,
height=\figureheight,
at={(0\figurewidth,0\figureheight)},
scale only axis,
hide axis,
xmin=-0.025,
xmax=1.025,
ymin=-7.5,
ymax=5.4,
]
\draw [arrows = {|-Stealth}] (0,-7) -- (1,-7) node [midway, fill=artcolorbg] {\footnotesize Time \({ t }\)};
\addplot [color=mycolor2,dashed,line width=2.0pt,forget plot]
  table[row sep=crcr]{%
0	-2.6\\
0.0690690690690691	-2.6\\
0.0700700700700701	1.4\\
0.129129129129129	1.4\\
0.13013013013013	-3.6\\
0.149149149149149	-3.6\\
0.15015015015015	4.4\\
0.229229229229229	4.4\\
0.23023023023023	-5.6\\
0.249249249249249	-5.6\\
0.25025025025025	2.4\\
0.399399399399399	2.4\\
0.4004004004004	-1.8\\
0.439439439439439	-1.8\\
0.44044044044044	0.3\\
0.64964964964965	0.3\\
0.650650650650651	-4\\
0.72972972972973	-4\\
0.730730730730731	-0.9\\
0.77977977977978	-0.9\\
0.780780780780781	1.2\\
0.94994994994995	1.2\\
0.950950950950951	-3\\
1	-3\\
};

\addplot [color=mycolor1,solid,line width=1.5pt,forget plot]
  table[row sep=crcr]{%
0	-2.63380426054901\\
0.0630630630630631	-2.63380426054901\\
0.0640640640640641	-1.12338223017403\\
0.0690690690690691	-1.12338223017403\\
0.0700700700700701	-0.00408820990606939\\
0.0780780780780781	-0.00408820990606939\\
0.0790790790790791	0.988619240359247\\
0.127127127127127	0.988619240359247\\
0.128128128128128	0.415818079350266\\
0.13013013013013	0.415818079350266\\
0.131131131131131	-2.24627487742504\\
0.132132132132132	-2.5414042195664\\
0.151151151151151	-2.5414042195664\\
0.152152152152152	2.0997683371987\\
0.153153153153153	2.76168466649047\\
0.154154154154154	4.2130356197868\\
0.215215215215215	4.2130356197868\\
0.216216216216216	4.11570432785356\\
0.225225225225225	4.11570432785356\\
0.226226226226226	2.50371650605971\\
0.23023023023023	2.50371650605971\\
0.231231231231231	-3.14048119214454\\
0.233233233233233	-3.14048119214454\\
0.234234234234234	-4.21941800026123\\
0.249249249249249	-4.21941800026123\\
0.25025025025025	-4.08516671569339\\
0.252252252252252	-4.08516671569339\\
0.253253253253253	2.38675960238296\\
0.33033033033033	2.38675960238296\\
0.331331331331331	2.16231682799792\\
0.399399399399399	2.16231682799792\\
0.4004004004004	0.477284082329938\\
0.401401401401401	0.076349912078143\\
0.405405405405405	0.076349912078143\\
0.406406406406406	-1.16789013244903\\
0.441441441441441	-1.16789013244903\\
0.442442442442442	-1.01741328520616\\
0.445445445445445	-1.01741328520616\\
0.446446446446446	-0.353152124954579\\
0.455455455455455	-0.353152124954579\\
0.456456456456456	0.357970863730537\\
0.647647647647648	0.357970863730537\\
0.648648648648649	-3.34033986908902\\
0.654654654654655	-3.34033986908902\\
0.655655655655656	-3.82325320867756\\
0.72972972972973	-3.82325320867756\\
0.730730730730731	-1.7742378314538\\
0.733733733733734	-1.7742378314538\\
0.734734734734735	-0.952767550086556\\
0.736736736736737	-0.952767550086556\\
0.737737737737738	-0.889323968782782\\
0.73973973973974	-0.889323968782782\\
0.740740740740741	-0.771043896304462\\
0.762762762762763	-0.771043896304462\\
0.763763763763764	-0.518575527165401\\
0.77977977977978	-0.518575527165401\\
0.780780780780781	-0.258193148860307\\
0.787787787787788	-0.258193148860307\\
0.788788788788789	1.12206261172342\\
0.946946946946947	1.12206261172342\\
0.947947947947948	0.454446282788835\\
0.948948948948949	-0.944663893356262\\
0.952952952952953	-0.944663893356262\\
0.953953953953954	-1.93434649033625\\
0.956956956956957	-1.93434649033625\\
0.957957957957958	-2.74737484182129\\
1	-2.74737484182129\\
};
\end{axis}
\end{tikzpicture}%

\vspace*{.5em}
\centering
\raisebox{4em}[0pt][0pt]{\ \begin{subfigure}[c]{.2\textwidth}%
\caption{Data:\\
\({ f^\dagger }\) (dashed black);\\
\({ f^\delta }\) (solid blue).}
\end{subfigure}}%
\quad%
\centering%
\setlength\figureheight{.25\linewidth}%
\setlength\figurewidth{.75\textwidth}%
\input{tikz/volterra_exp_data.tikz}
 \caption{TV~denoising for the Volterra operator~\eqref{eq:op-Volterra-convolution} with exponential kernel~\({ x \mapsto \exp(-t/10) }\).}\label{fig:volterra-exp}
\end{figure}

\section{A parameter-identification problem for a semilinear parabolic PDE}
\label{sec:param-id-problem}

\subsection{Setup}

\enlargethispage{\baselineskip} 

We now provide an application of subgradient-based Lavrentiev regularisation~\eqref{eq:subgradient-lavrentiev-reg} to a class of parameter-identification problems for semilinear parabolic PDEs.
To that end, let \({ \Omega \subset \R^d }\) be a Lipschitz domain,
\({ I \coloneqq (0, 1) }\) denote the unit interval,
\({ g \in \textnormal{L}^\infty(\Omega) }\) be some given initial data, and
\begin{equation} \label{eq:PDE-nonlinearity-growth}
\varphi(y) \coloneqq \sign(y) \abs{ y }^{ \lambda } \qquad \text{for some } \lambda > 0
\end{equation}
be a strictly monotone nonlinearity. With \({ \X \cong \X^* }\) denoting the Hilbert space \({ \textnormal{L}^2(I \times \Omega) }\), we then define \({ \mathcal{A} \colon \X \to \X }\) to be the operator that maps \({ u \in \X }\) to the weak solution of the PDE
\begin{equation}\label{eq:PDE}
  \left\{\begin{aligned}
    y_t - \upDelta y + \varphi(y) &= u && \text{ in } I \times \Omega; \\
    \partial_n y &= 0 && \text{ on } I \times \partial\Omega; \\
    y(0, \cdot) &= g && \text{ on } \Omega.
  \end{aligned}\right.
\end{equation}
Note that this equation can alternatively be written as
the gradient flow
\begin{equation} \label{eq:gradient-flow}
y_t + \partial G(y) \ni u,
  \qquad\qquad
  y(0) = g,
\end{equation}
where \({ G \colon \textnormal{L}^2(\Omega) \to (- \infty, \infty] }\) is the convex and lower-semicontinuous functional
\begin{equation*}
  G(y) \coloneqq 
  \begin{cases}
    \displaystyle\int_\Omega \left( \tfrac{1}{2}\abs{\nabla y}^2 + \tfrac{1}{\lambda + 1} \abs{y}^{\lambda + 1 }\right) \dee x
    & \text{ if } y \in \textnormal{H}^1(\Omega)\cap \textnormal{L}^{\lambda + 1}(\Omega); \\[1ex]
    +\infty & \text{ else.}
  \end{cases}
\end{equation*}
Consequently, the standard theory on gradient flows
in Hilbert spaces (see \autocite{Bre1973a} or~\autocite[Theorems~4.2, 4.5 and~4.11]{Bar2010a}) implies that the PDE~\eqref{eq:PDE}
has a unique (mild/strong) solution in~\({ \X }\) that depends continuously on~\({ u }\) with respect to strong convergence in~\({ \X }\).
Put differently, the operator~\({ \mathcal{A} }\) is well-defined and continuous.
In addition, analogous to the standard proof of uniqueness for~\eqref{eq:gradient-flow}, one finds the comparison principle
\begin{equation} \label{eq:comparison-principle}
y \leq z \quad \text{a.e.\@ in } I \times \Omega \qquad \text{whenever} \qquad u \leq v \text{ and } g \leq h \quad\text{ a.e.,}
\end{equation}
where \({ y = \mathcal{A}u }\) and \({ z = \mathcal{A}v }\) solve~\eqref{eq:PDE} with initial data \({ g }\) and \({ h }\), respectively.
Moreover, in the case \({ g = h }\), one easily obtains the equality
\begin{equation*}
\dualitypairing{ y - z }{ u - z } = \tfrac{ 1 }{ 2 } \int_{ \Omega } \bigl( y(1, \cdot) - z(1, \cdot) \bigr)^2 \dee x + \int_{ I } \int_{ \Omega } \left( \abs{ \nabla(y - z) }^{ 2 } + (\varphi(y) - \varphi(z))(y - z) \right) \dee x \dee t,
\end{equation*}
which in view of the strict monotonicity of~\({ \varphi }\) shows that~\({ \mathcal{A} }\) is strictly monotone.

We now consider the parameter-identification problem of solving \({ \mathcal{A}(\overline{u}) = y^\delta }\) given possibly noisy data~\({ y^\delta \in \X }\).
For the stable solution of this problem,
we assume that the true solution associated with the noise-free data~\({ y^\dagger }\) has the form
\begin{equation*}
\overline{u}^\dagger = u^\dagger + f,
\end{equation*}
where \({ f \in \textnormal{L}^\infty(I \times \Omega) }\) is a known function
and \({ u^\dagger \in \textnormal{BV}(I\times\Omega) }\) is an unknown function that is constant in the spatial variable---that is, there exists \({ v^\dagger \in \textnormal{BV}(I)}\) such that
\begin{equation*}
u^\dagger(t, x) = v^\dagger(t) \qquad \text{for a.e.\@ } (t, x) \in I \times \Omega.
\end{equation*}
In other words, \({ u^\dagger = \imath(v^\dagger) }\), where
\begin{equation*}
\imath \colon \textnormal{L}^2(I) \to \X
\end{equation*}
is the natural embedding.
With \({ \mathcal{R} \colon \X \to [0, \infty] }\) defined by
\begin{equation*}
  \mathcal{R}(u) \coloneqq
  \begin{cases}
    \seminorm{u}_{\textnormal{TV}(I \times \Omega)} & \text{ if } u \in \textnormal{BV}(I \times \Omega)\cap \range \imath ; \\
    +\infty &\text{ else,}
  \end{cases}
\end{equation*}
we see that subgradient-based Lavrentiev regularisation~\eqref{eq:subgradient-lavrentiev-reg} adapted to this context requires finding the solution~\({ u_ \alpha^\delta }\) of the inclusion
\begin{equation}\label{eq:Lav_parid}
  \mathcal{A}(u_\alpha^\delta + f) + \alpha \partial\mathcal{R}(u_\alpha^\delta) \ni y^\delta.
\end{equation}

In order to show that this is a well-posed regularisation method,
we need to verify that \Cref{assumptions} is satisfied.
The requirements on the space~\({ \X }\) and the regulariser~\({ \mathcal{R} }\) are directly satisfied,
and, as discussed above, the operator~\({ \mathcal{A} }\) is both continuous and strictly monotone on~\({ \X }\).
Accordingly, it remains to establish the coercivity and the
growth conditions in \Cref{assumptions}~\ref{assumptions:coercivity} and~\ref{assumptions:growth}.

\subsection{Verification of the conditions}

We assume for simplicity that \({ \abs{ \Omega } = 1 }\).
Then \({ \norm{\imath v}_{ \X} = \norm{v}_{ \textnormal{L}^2(I)}}\) and \({ \mathcal{R}(\imath v) = \seminorm{v}_{ \textnormal{TV}(I)} }\) for all~\({ v \in \textnormal{L}^2(I) }\).
By compactness of \({ \textnormal{BV}(I) }\) in \({ \textnormal{L}^2(I) }\),
it is then evident that the sublevel set 
\begin{equation*}
\widetilde{W} \coloneqq \Set[\big]{ v \in \textnormal{L}^2(I) \given \norm{ v }_{ \textnormal{L}^2(I) } + \seminorm{ v }_{ \textnormal{TV}(I)} \leq C }
\end{equation*}
is compact for every~\({ C > 0 }\),
and thus the same is true for
\begin{equation*}
W \coloneqq \imath \widetilde{W} = \Set[\big]{ u \in \X \given \norm{ u }_{ \X } + \mathcal{R}(u) \leq C },
\end{equation*}
which shows that \Cref{assumptions}~\ref{assumptions:coercivity} holds.

Now let \({ C > 0 }\) and define
\[
  U_C \coloneqq \Set[\big]{ u \in \X \given \norm{u-u^\dagger}_{\X} = 1 \; (= t_0) \text{ and } \mathcal{R}(u) \leq C}.
\]
Moreover, for \({u \in U_C }\) and \({r > 0 }\), let
\({u_r \coloneqq ru + (1-r)u^\dagger }\).
We then need to prove that
\begin{equation*}
\inf_{u\in U}\dualitypairing*{\mathcal{A}(u_r)}{u - u^\dagger} \to \infty \quad \text{as } r \to \infty.
\end{equation*}
Introducing
\[
  V \coloneqq \Set[\big]{ u \in L^2(I) \given \norm{v-v^\dagger}_{L^2(I)} = 1 \text{ and } \seminorm{v}_{\textnormal{TV}} \leq C}
\]
and
\[
  v_r \coloneqq rv + (1-r)v^\dagger
\]
for \({v \in V}\) and \({r > 0}\),
we see that this is equivalent to showing that
\begin{equation*}
\inf_{v\in V}\dualitypairing*{\mathcal{A}(\imath v_r)}{\imath(v - v^\dagger)} \to \infty \quad \text{as } r \to \infty.
\end{equation*}
The proof of this result will be based on comparison principles and relating
solutions of the PDE~\eqref{eq:PDE} with the flow of an ODE.

\begin{remark}
In the remaining text we identify \({ v \in \textnormal{BV}(I) }\) with any of its \textit{good representatives}~\({ \widetilde{ v } }\),
for which the measure-theoretic TV~\eqref{eq:total-variation-regulariser} coincides with the classical pointwise TV.
Here \({ \widetilde{ v } \colon I \to \R }\) is pointwise defined and assumes a value between (typically the average of) the one-sided essential limits
\begin{equation*}
v_{ \textnormal{left}}(t) \coloneqq \esslim_{ s \nearrow t } v(s) \qquad \text{and} \qquad v_{ \textnormal{right}}(t) \coloneqq \esslim_{ s \searrow t } v(s)
\end{equation*}
for every~\({ t \in I }\), so that \({ \widetilde{ v } }\) agrees with the essential pointwise limit of \({ v }\) wherever \({ v }\) is essentially continuous.
\end{remark}

To this end, for every \({w\in \textnormal{L}^1(I)}\) let \({ z_{ w } \colon I \to \R }\) denote the unique mild solution of the ODE
\begin{equation} \label{eq:ODE}
z_{w}' + \varphi(z_{w}^{}) = w,
  \qquad
  z_{w}^{}(0) = 0
\end{equation}
provided by the theory of gradient flows~\autocite[Theorem~4.11]{Bar2010a},
where we remember that \({ \varphi }\) is the subgradient of \({ \frac{1}{p + 1} \abs{  }^{ p + 1 }  }\).
Observe also that if \({ w }\) is continuous,
the Peano existence theorem would have guaranteed a global (since \({ \varphi }\) is weakly coercive) solution in~\({ I }\)---in spite of \({ (t, z) \mapsto - \varphi(z) + w(t) }\) failing to be locally Lipschitz continuous in~\({ z }\).
Moreover, uniqueness follows from the fact that \({ - \varphi }\) is decreasing.

We then show that condition~\eqref{eq:coercivity-growth} for the PDE~\eqref{eq:PDE} may be reduced to that of the ODE~\eqref{eq:ODE}.

\enlargethispage{\baselineskip} 

\begin{lemma}\label{thm:PDEestimate}
There exists a constant~\({ K > 0 }\) independent of \({ r > 0 }\) and \({ v \in V }\) such that
\begin{equation*}
\abs{z_{r(v-v^\dagger)}(t) - \mathcal{A}(\imath v_r)(t, x)} \leq K
\end{equation*}
for almost every \({ (t, x) \in I \times \Omega}\).
\end{lemma}

\begin{proof}
Define
\begin{equation*}
\begin{aligned}
      f^+ &:= \max \Set{ 0, \esssup (f+v^\dagger) }, & g^+ &:= \max \Set{ 0, \esssup g },\\
      f^- &:= \min \Set{ 0, \essinf (f+v^\dagger) },& \text{ and }\quad g^- &:= \min \Set{ 0, \essinf g },
    \end{aligned}
\end{equation*}
and for fixed \({ r > 0 }\) and \({ v \in V }\) denote by \({ y_{r, v}^\pm }\) the solutions of the PDEs
 \begin{equation*}
 \left\{\begin{aligned}
      \partial_t y_{r,v}^\pm - \upDelta y_{r,v}^\pm + \varphi(y_{r,v}^\pm)
      &= f^\pm+ r(v(t)-v^\dagger(t)) && \text{ in } I\times \Omega; \\[1ex]
      \partial_n y_{r,v}^\pm &= 0 && \text{ on } I \times \partial\Omega; \\[1ex]
      y_{r,v}^\pm(0,\cdot) &= g^\pm && \text{ on } \Omega.
    \end{aligned}\right.
\end{equation*}
Recall here that \({v^\dagger \in \textnormal{BV}(I) \subset \textnormal{L}^\infty(I)}\),
\({f \in \textnormal{L}^\infty(I\times\Omega)}\), and \({g \in \textnormal{L}^\infty(\Omega)}\),
which implies that \({f^\pm}\) and \({g^\pm}\) are finite numbers.
Then the comparison principle~\eqref{eq:comparison-principle} yields that
\begin{equation*}
y_{r, v}^-(t, x) \leq \mathcal{A}(\imath v_r)(t, x) \leq y_{r, v}^+(t, x)
\end{equation*}
for almost every \({ (x, t) \in I \times \Omega }\).
In addition, it is straightforward to see that \({ y_{r, v}^\pm }\) actually are constant with respect to the \({ x }\)-variable
and can be written in the form \({ y_{r, v}^\pm(t, x) = z_{r,v}^\pm(t) }\),
where the functions \({ z_{r, v}^\pm \colon I \to \R }\) solve the ODEs
\begin{equation*}
    \frac{\dee z_{r, v}^\pm}{\dee t} + \varphi(z_{r, v}^\pm) = f^\pm + r(v-v^\dagger),
    \qquad
    z_{r ,v}^\pm(0) = g^\pm.
\end{equation*}
  Thus
\begin{equation} \label{eq:ODE-upper-bnd}
    \frac{\dee}{\dee t}(z_{r, v}^+ - z_{r, v}^-)
    = f^+ - f^- - \bigl(\varphi(z_{r, v}^+) - \varphi(z_{r, v}^-)\bigr)
    \leq f^+ - f^-,
\end{equation}
  which implies that
\begin{equation*}
    \abs[\big]{z_{r, v}^+(t) - z_{r, v}^-(t)} \leq g^+ - g^- + t(f^+-f^-)
\end{equation*}
  for all \({ t \in I }\).
  Since we also have the ODE-variant
\begin{equation*}
    z_{r, v}^-(t) \leq z_{r, v}^{}(t) \leq z_{r, v}^+(t)
\end{equation*}
of the comparison principle~\eqref{eq:comparison-principle},
we obtain the assertion with \({ K = g^+ - g^- + f^+ - f^- }\).
\end{proof}

As a consequence of \cref{thm:PDEestimate}, it suffices to show that
\begin{equation*}
\adjustlimits \lim_{r\to\infty} \inf_{v \in V} \dualitypairing*{z_{r, v}}{v-v^\dagger}_{ \textnormal{L}^2(I)} = +\infty.
\end{equation*}
On this path, we estimate
\begin{equation} \label{eq:ODE-growth-estimate}
  \inf_{v\in V} \dualitypairing*{z_{r, v}}{rv - v^\dagger}_{ \textnormal{L}^2(I)}
  \geq \inf_{v\in V} \dualitypairing*{z_{r, v}}{rv}_{ \textnormal{L}^2(I)} - \sup_{v\in V} \norm{z_{r, v}}_{ \textnormal{L}^1(I)} \norm{v^\dagger}_{ \textnormal{L}^\infty(I)},
\end{equation}
and note that for all \({ v \in V }\) we have
\begin{equation} \label{eq:supremum-bound}
\norm{ v }_{ \textnormal{L}^\infty(I) } \leq \norm{ v }_{ \textnormal{L}^1(I) } + \seminorm{v}_{ \textnormal{TV}(I)} \leq \norm{v}_{ \textnormal{L}^2(I)} + \seminorm{v}_{ \textnormal{TV}(I)} \leq 1 + C.
\end{equation}
Thus we may estimate
\begin{equation*}
-(1 + C)rt \leq z_{-r, (1 + C)}(t) \leq z_{r,v}(t) \leq z_{r, (1 + C)}(t) \leq (1 + C)rt
\end{equation*}
for all \({ r > 0 }\) and \({ v \in V }\), similarly as~\eqref{eq:ODE-upper-bnd}, which leads to
\begin{equation*}
\sup_{v\in V} \norm{z_{r, v}}_{ \textnormal{L}^1(I)} \leq \tfrac{1}{2} r(1 + C),
\end{equation*}
expressing that the last term in~\eqref{eq:ODE-growth-estimate} has at most linear rate in~\({ r }\).
Since moreover
\begin{equation*}
  \dualitypairing*{z_{r, v}}{rv}_{ \textnormal{L}^2(I)}
  = \tfrac{1}{2}\abs*{z_{r, v}(1)}^2 + \dualitypairing*{\varphi(z_{r,v})}{z_{r,v}}_{ \textnormal{L}^2(I)}
  \geq \dualitypairing*{\varphi(z_{r, v})}{z_{r,v}}_{ \textnormal{L}^2(I)},
\end{equation*}
it is sufficient to show that
\begin{equation*}
  r \mapsto \inf_{v \in V}\dualitypairing*{\varphi(z_{r,v})}{z_{r,v}}_{ \textnormal{L}^2(I)}
\end{equation*}
grows superlinearly.

\begin{lemma}\label{thm:PDEgrowth}
There exist numbers \({ c > 0 }\) and \({ \tau > 0 }\), such that for each \({ v \in V }\)
there exists an interval \({ J_v \coloneqq \left( t_v - 2 \tau, t_v + 2 \tau \right) \subseteq I }\)
with either \({ v(t)-v^\dagger(t) \geq c }\) for all \({ t \in J_v }\)
or \({ v(t)-v^\dagger(t) \leq - c }\) for all~\({ t \in J_v }\).
\end{lemma}

\begin{proof}
Due to the weak compactness of \({ V }\) in \({ \textnormal{L}^2(I) }\),
the convex and lower-semicontinuous function \({ v \mapsto \norm{ v-v^\dagger }_{ \textnormal{L}^1(I) } }\) attains its minimum on~\({ V }\),
which must be nonzero in light of the definition of~\({ V }\).
Thus there exists an \({ a > 0 }\) such that \({ \norm{ v-v^\dagger }_{ \textnormal{L}^1(I) } \geq 2a }\) for all \({ v \in V }\).
Since \({ \abs{ I } = 1 }\),
we find from~\eqref{eq:supremum-bound} that
\begin{equation*}
\abs*{ \Set[\big]{ t \in I \given \abs{ v(t)-v^\dagger(t) } \geq a } } \geq \frac{ a }{ \norm{ v-v^\dagger }_{ \textnormal{L}^\infty(I)} }  \geq \frac{ a }{ 1 + C + \seminorm{v^\dagger}_{\textnormal{BV}}} \eqqcolon 2m.
\end{equation*}
Consequently, at least one of the sets \({ \Set{ t \given  v(t) -v^\dagger(t)\geq a } }\) and \({ \Set{ t \given v(t)-v^\dagger(t) \leq - a } }\) has measure greater than or equal to~\({ m }\),
and without loss of generality we may assume the former.

Denote by \({ A_k }\),
for \({ k }\) in some countable index set~\({ K }\),
the connected components of the (possibly larger) set \({ \Set{ t \given v(t) \geq a / 2 } }\),
and let \({ K' \subseteq K }\) be the set of indices for which \({ \sup_{ t \in A_k} v(t) \geq a }\).
Note that for any distinct \({ k_1, k_2 \in K' }\), the total variation of \({ v }\) on the convex hull of \({ A_{k_1} \cup A_{k_2} }\) is at least~\({ a }\),
because \({ v }\) must pass from \({ a }\) to \({ a/2 }\) and back to \({ a }\) again.
Since \({ \seminorm{v}_{ \textnormal{TV}(I)} \leq C }\),
this implies that \({ \abs{ K' } \leq \frac{ C }{ a } + 1 \eqqcolon n }\),
which also covers the trivial case with \({ v }\) being constant (\({ C = 0 }\)).
On the other hand, we have that \({ \sum_{ k \in K' } \abs{ A_k } \geq m }\) by construction.
Thus there exists at least one connected component~\({ A_k }\) for which~\({ \abs{ A_k } \geq m / n }\).

Hence, the assertion holds with \({ c = a / 2 }\) and \({ \tau = m / n }\), both being independent of~\({ v \in V }\).
\end{proof}

\begin{lemma}\label{thm:wr}
Let \({ c }\) and \({ \tau }\) be as in \cref{thm:PDEgrowth}
and denote by \({ w_r \colon \R \to \R }\) the solution of the ODEs
\begin{equation*}
    w_r' = r - \varphi(w_r^{}),
    \qquad
    w_r^{}(0) = 0.
\end{equation*}
  Then
  \begin{equation}\label{eq:w_estimate}
    \inf_{v \in V}\dualitypairing*{\varphi(z_{r, v})}{z_{r, v}}_{ \textnormal{L}^2(I)} \geq 
    \tau \, w_{rc}(\tau) \, \varphi(w_{rc}(\tau))
  \end{equation}
for all~\({ r > 0 }\).
\end{lemma}

\begin{proof}
Let \({ v \in V }\) and \({ c }\), \({ \tau }\), and \({ J_v = (t_v - 2 \tau, t_v + 2 \tau) }\) be as in \cref{thm:PDEgrowth}.
Then either \({ v(t) \geq c }\) for all \({ t \in J_v }\) or \({ v(t) \leq - c }\) for all \({ t \in J_v }\).
We assume and prove the result for the former case;
the argument for the other situation is similar.
In particular, we find from~\eqref{eq:ODE} that \({ z_{r, v}' \geq cr - \varphi(z_{ r, v}^{}) }\) on \({ J_v }\) also.

If \({ z_{r, v}(t_v) \geq 0 }\), then \({ z_{r, v}(t) \geq \widetilde{ w }(t) }\) for all \({ t \in [t_v, t_v + 2 \tau] }\),
where \({ \widetilde{ w } }\) solves the initial-value problem
  \begin{equation}\label{eq:tildew}
    \widetilde{w}' = cr - \varphi(\widetilde{w}),
    \qquad
    \widetilde{w}(t_v) = 0.
  \end{equation}
Moreover, \({ \widetilde{ w } }\) is nonnegative and increasing on~\({ [t_v, t_v + 2 \tau] }\).
Since \({ \varphi }\) is increasing and \({ \varphi(0) = 0 }\), this leads to
\begin{equation*}
\dualitypairing*{ z_{r, v} }{ \varphi( z_{r, v}) }_{ \textnormal{L}^2(I)}
\geq \int_{ t_v + \tau }^{ t_v + 2 \tau } \widetilde{ w }(t) \, \varphi(\widetilde{ w }(t)) \dee t
\geq \tau \, \widetilde{ w }(t_v + \tau) \, \varphi( \widetilde{ w }(t_v + \tau)).
\end{equation*}
Now we note that \({ \widetilde{ w }(t_v + \tau) = w_{rc}(\tau) }\), which proves~\eqref{eq:w_estimate}.

If, however, \({ z_{r, v}(t_v) \leq 0 }\), then \({ z_{r, v}(t) \leq \widetilde{ w }(t) }\) for all~\({ t \in [t_v - 2 \tau, t_v] }\),
where \({ \widetilde{ w } }\) again solves~\eqref{eq:tildew}, but now considered as a final-value problem.
Similarly as before, we find that
\begin{align*}
\dualitypairing*{ z_{r, v} }{ \varphi( z_{r, v}) }_{ \textnormal{L}^2(I)}
\geq \int_{ t_v - 2 \tau }^{ t_v - \tau } \widetilde{ w }(t) \, \varphi(\widetilde{ w }(t)) \dee t
&\geq \tau \, \widetilde{ w }(t_v - \tau) \, \varphi(\widetilde{ w }(t_v - \tau)) \\[1ex]
&= \tau \, w_{rc}(-\tau) \, \varphi(w_{rc}(- \tau)),
\end{align*}
and this establishes~\eqref{eq:w_estimate} by noting that \({ w_{rc}(- \tau) = - w_{rc}(\tau) }\) and \({ \varphi(w_{rc}(- \tau)) = - \varphi(w_{rc}(\tau)) }\) due to antisymmetry of~\({ \varphi }\).
\end{proof}

As a consequence of \cref{thm:wr},
it is sufficient to show that the function \({ r \mapsto w_r(\tau) \,\varphi(w_r(\tau)) }\) grows superlinearly as \({ r \to \infty }\).
To that end, we distinguish between the two cases \({ 0 < \lambda \leq 1 }\) and \({ \lambda > 1 }\) concerning the growth~\eqref{eq:PDE-nonlinearity-growth} of~\({ \varphi }\).

If \({ 0 < \lambda \leq 1 }\), then \({ \varphi(z) \leq 1 + z }\) for all~\({ z \geq 0 }\).
Thus \({ w_r }\) is bounded below by the solution of the ODE
\begin{equation*}
y_r' = r - 1 - y_r, \qquad y_r(0) = 0,
\end{equation*}
and a brief computation reveals that \({ y_r(t) = (r-1) \left( 1 - \textnormal{e}^{ - t }  \right)  }\).
Accordingly,
\begin{equation*}
w_r(\tau) \, \varphi(w_r(\tau)) = w_r(\tau)^{1 + \lambda} \geq y_r(\tau)^{1 + \lambda} \gtrsim r^{1 + \lambda}
\end{equation*}
as \({ r \to + \infty }\), proving the superlinear growth.

Next we consider the case \({ \lambda > 1 }\),
where we first note that \({ w_r }\) is bounded above by the stationary point \({ \varphi^{-1}(r) = r^{1 / \lambda} }\).
Since \({ r \mapsto \varphi(r) = r^\lambda }\) is convex on~\({ [0, + \infty) }\),
we may estimate
\begin{equation*}
\frac{ r - z^\lambda }{ r^{1 / \lambda} - z } = \frac{ \varphi(r^{1 / \lambda}) - \varphi(z) }{ r^{1 / \lambda} - z } \geq \frac{ \varphi(r^{1 / \lambda}) - \varphi(0) }{ r^{1 / \lambda} - 0 } = r^{ \frac{ \lambda - 1 }{ \lambda }}
\end{equation*}
for \({ 0 \leq z \leq r^{1 / \lambda} }\), which becomes
\begin{equation*}
r - z^\lambda \geq r^{ \frac{ \lambda - 1 }{ \lambda } } \left( r^{1 / \lambda} - z \right).
\end{equation*}
As such, \({ w_r }\) is bounded below by the solution of the ODE
\begin{equation*}
y_r' = r^{ \frac{ \lambda - 1 }{ \lambda } } \left( r^{1 / \lambda} - y_r \right), \qquad y_r(0) = 0.
\end{equation*}
One sees that \({ y_r(t) = r^{1 / \lambda} \bigl( 1 - \exp \bigl( - r^{ \frac{ \lambda - 1 }{ \lambda }} t \bigr) \bigr) }\).
Since the exponential term in \({ y_r }\) tends, for all fixed \({ t > 0 }\), to \({ 0 }\) as \({ r \to +\infty }\),
we thus obtain that
\begin{equation*}
w_r(\tau) \, \varphi(w_r(\tau)) = w_r(\tau)^{1 + \lambda} \geq y_r(\tau)^{1 + \lambda} \gtrsim r^{1 + \frac{ 1 }{ \lambda }}
\end{equation*}
for sufficiently large~\({ r }\),
again resulting in superlinear growth.

Collecting all the results above,
it follows that subgradient-based
Lavrentiev regularisation of the form~\eqref{eq:Lav_parid}
is a well-posed regularisation method.

\subsection{Numerical formulation and example}

We now reformulate the regularisation problem~\eqref{eq:Lav_parid} in order to make it numerically tractable.
To that end, let \({ \pi \coloneqq \imath^* \colon \X \to \textnormal{L}^2(I) }\)
be the projection
\begin{equation*}
  \pi u (t) = \frac{1}{\abs{\Omega}} \int_\Omega u(t, x) \dee x,
\end{equation*}
and define \({ \mathcal{S} \colon \textnormal{L}^2(I) \to [0, \infty] }\) by
\begin{equation*}
  \mathcal{S}(v) \coloneqq 
  \begin{cases}
    \seminorm{v}_{\textnormal{TV}(I)} & \text{ if } v \in \textnormal{BV}(I); \\
    +\infty &\text{ else.}
  \end{cases}
\end{equation*}
Then we have that \({ \mathcal{R} = S \circ \pi + \chi_{ \range \imath} }\),
where
\begin{equation*}
\chi_{ \range \imath}(u) \coloneqq
\begin{cases}
0 & \text{ if } u \in \range \imath;\\
\infty & \text{ else,}
\end{cases}
\end{equation*}
\clearpage\noindent 
is the indicator function on~\({ \range \imath }\).
Since \({ \partial \chi_{ \range \imath}(u) = \ker \pi }\),
subdifferential calculus~\autocite[Thm.~16.47]{BauCom2017a} gives that
\begin{equation*}
  \partial\mathcal{R}(u) =
  \begin{cases}
    \imath\, \partial\mathcal{S}(\pi u) + \ker \pi & \text{ if } u \in \range \imath;\\
    \emptyset & \text{ else.}
  \end{cases}
\end{equation*}
Instead of solving for \({ u_\alpha^\delta \in \X }\),
we may therefore rewrite this as a problem for \({ v_\alpha^\delta \in \textnormal{L}^2(I) }\) with \({ \imath v_\alpha^\delta = u_\alpha^\delta }\).
Thus we arrive at the equivalent reformulation of~\eqref{eq:Lav_parid} as
\({ u_\alpha^\delta = \imath v_\alpha^\delta }\), where \({ v_\alpha^\delta \in \textnormal{L}^2(I) }\) solves the inclusion
\begin{equation*}
  \mathcal{A}(\imath v_\alpha^\delta + f) + \alpha \, \imath\,\partial\mathcal{S}(v_\alpha^\delta) + \ker \pi \ni y^\delta,
\end{equation*}
which can again be rewritten as
\begin{equation*}
  -\pi\bigl(\mathcal{A}(\imath v_\alpha^\delta + f) - y^\delta\bigr) \in \alpha \partial\mathcal{S}(v_\alpha^\delta),
\end{equation*}
using that \({ \pi \circ \imath = \identityoperator_{ \textnormal{L}^2(I)} }\).
Or, introducing the auxiliary variable \({ \xi_\alpha^\delta \in \partial\mathcal{S}(v_\alpha^\delta) }\),
we obtain the system
\begin{equation}\label{eq:reformulation_parid}
\left\{\begin{aligned}
    \pi \mathcal{A}(\imath v_\alpha^\delta + f) + \alpha \, \xi_\alpha^\delta &= \pi y^\delta; \\[1ex]
    \xi_\alpha^\delta & \in \partial\mathcal{S}(v_\alpha^\delta).
  \end{aligned}
  \right.
\end{equation}

\begin{figure}[h!]
\begin{subfigure}[c]{.49\textwidth}%
\centering%
\setlength\figureheight{.99\linewidth}%
\setlength\figurewidth{.99\linewidth}%
\includegraphics[width=\linewidth]{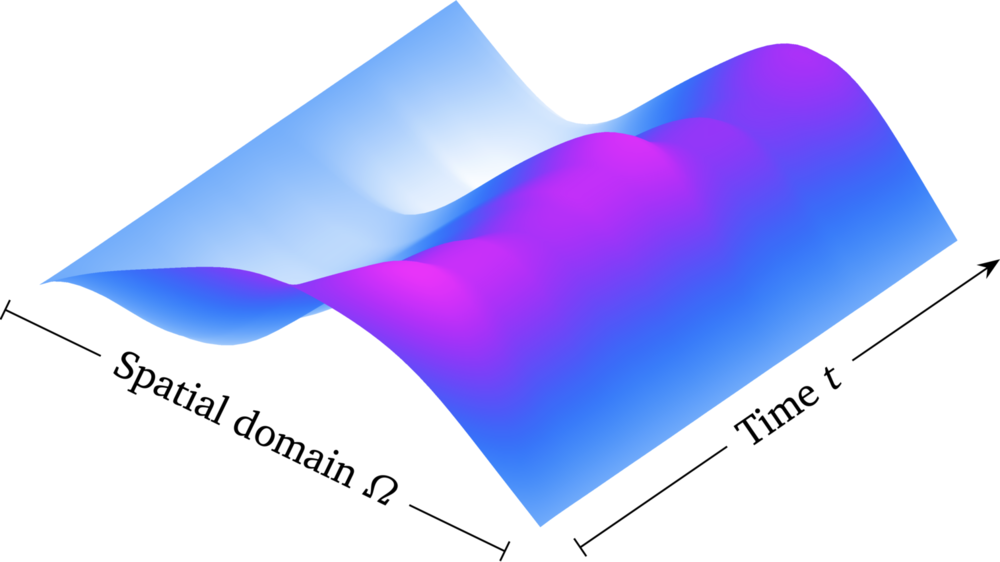}%
\caption{Exact data: \({ y^\dagger }\).}%
\vspace*{2em}%
\end{subfigure}
\begin{subfigure}[c]{.49\textwidth}%
\centering%
\setlength\figureheight{.99\linewidth}%
\setlength\figurewidth{.99\linewidth}%
\includegraphics[width=\linewidth]{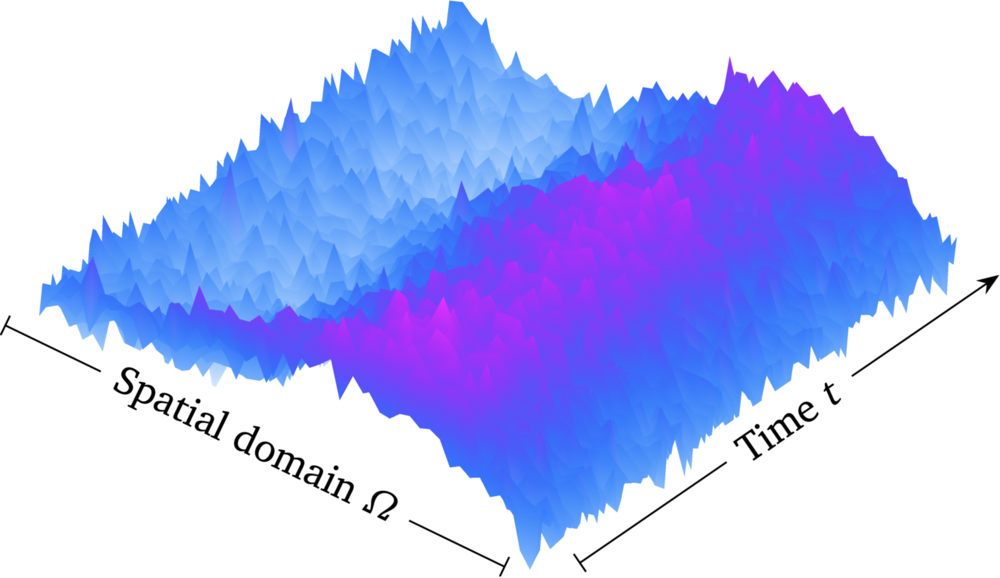}%
\caption{Noisy data: \({ y^\delta }\).}%
\vspace*{2em}%
\end{subfigure}

\captionsetup{format=hang}
\vspace*{0em}%
\centering%
\raisebox{5em}[0pt][0pt]{\ \begin{subfigure}[c]{.22\textwidth}%
\caption{Projected data:\\
\({ \pi y^\dagger }\) (dashed black); \\
\({ \pi y^\delta }\) (solid gray); \\
\({ \pi y_\alpha^\delta }\) (solid red).}
\end{subfigure}}%
\quad%
\centering%
\setlength\figureheight{.25\linewidth}%
\setlength\figurewidth{.75\textwidth}%
\begin{tikzpicture}

\colorlet{mycolor1}{WildStrawberry}%
\colorlet{mycolor2}{artcolortext}%

\begin{axis}[
width=\figurewidth,
height=\figureheight,
at={(0\figurewidth,0\figureheight)},
scale only axis,
hide axis,
xmin=-0.025,
xmax=1.025,
ymin=-.04,
]
\addplot [color=Gray,solid,line width=1.0pt,forget plot] table {python/PDE_projected_data_noisy.txt};
\addplot [color=mycolor1,solid,line width=1.5pt,forget plot] table {python/PDE_projected_data_recon.txt};
\addplot [color=mycolor2,dashed,line width=2.0pt,forget plot] table {python/PDE_projected_data.txt};
\draw [arrows = {|-Stealth}] (0,0.02) -- (1,0.02) node [midway, fill=artcolorbg] {\footnotesize Time \({ t }\)};
\end{axis}%
\end{tikzpicture}%
%

\vspace*{1em}
\centering%
\raisebox{5em}[0pt][0pt]{\begin{subfigure}[c]{.22\textwidth}%
\caption{Solutions:\\
\({ v^\dagger }\) (dashed black);\\
\({ v_\alpha^\delta }\) (solid red).}
\end{subfigure}}%
\quad%
\centering%
\setlength\figureheight{.2\linewidth}%
\setlength\figurewidth{.75\textwidth}%
\begin{tikzpicture}

\colorlet{mycolor1}{WildStrawberry}%
\colorlet{mycolor2}{artcolortext}%

\begin{axis}[
width=\figurewidth,
height=\figureheight,
at={(0\figurewidth,0\figureheight)},
scale only axis,
hide axis,
xmin=-0.025,
xmax=1.025,
ymin=-2.2,
ymax=2.2,
]
\addplot [color=mycolor2,dashed,line width=2.0pt,forget plot] table {python/PDE_control.txt};
\addplot [color=mycolor1,solid,line width=1.5pt,forget plot] table {python/PDE_control_recon.txt};
\end{axis}%
\end{tikzpicture}%
%
\caption{TV~denoising in the parameter-identification problem~\eqref{eq:Lav_parid} with the choice~\eqref{eq:PDE-example} in~\eqref{eq:PDE}. By~\({ y_\alpha^\delta }\) we mean the solution of~\eqref{eq:PDE} corresponding to the reconstruction~\({ u_\alpha^\delta = \imath v_\alpha^\delta }\).}
\label{fig:PDE}
\end{figure}

In order to solve~\eqref{eq:reformulation_parid},
we employ the classical taut-string algorithm~\autocite{DavKov2001a,Gra2007a,Con2013a} for~\({ v_\alpha^\delta }\) on~\({ I }\),
noting that at each step in the algorithm,
one has to solve the PDE~\eqref{eq:PDE} and perform projection onto~\({ \textnormal{L}^2(I) }\).
As an example, we perform TV denoising in~\eqref{eq:PDE} with
\begin{align} \label{eq:PDE-example}
\begin{aligned}
\varphi(y) &= y^3, \\[1ex]
f(x) &= 20\left(\sin(2\uppi x)\right)^3 + 2\sqrt{ x }, \\[1ex] \text{and } g(x)& = \tfrac{ 1 }{ 2 } \sin (\uppi x).
\end{aligned}
\end{align}
The typical performance under moderate gaussian noise in \({ I \times \Omega }\) can be seen in \cref{fig:PDE}.
The PDE-solution~\({ y^\dagger }\) has been computed (from~\({ v^\dagger }\)) using a much finer resolution in order to avoid inverse crimes.
Observe that we essentially only work with the projected noisy data in~\({ I }\) rather than the spatially-dependent data on~\({ I \times \Omega }\),
which explains the adequate reconstructions.

\begin{wrapfigure}{r}{0.5\textwidth}
\vspace*{2em} 
\centering%
\setlength\figureheight{.8\linewidth}%
\setlength\figurewidth{.8\linewidth}%
\begin{tikzpicture}

\colorlet{mycolor1}{artcolor}%
\colorlet{mycolor2}{WildStrawberry}%

\begin{loglogaxis}[
width=\figurewidth,
height=\figureheight,
at={(0\figurewidth,0\figureheight)},
scale only axis,
axis lines=left,
xmin=8e-5,
xmax=0.03,
ymax=.4,
xlabel={\({ \delta }\)},
every axis x label/.style={at={(ticklabel* cs:1)},anchor=west},
legend style={draw=none,fill=none,at={(0.6,0.2)},anchor=west,column sep=1ex}
]
\addplot [color=mycolor1,solid,line width=1.2pt,mark=+] table [x index=0, y index=2] {python/convergence_plot.txt};
\addlegendentry{\({  \textnormal{L}^2 }\) error};
\addplot [color=mycolor1,no marks,dotted,line width=1pt, forget plot] table [x=noise, y={create col/linear regression={y=l2,variance list={70,50,50,50,200,200,400,1000}}}]
        {python/convergence_plot.txt}
            coordinate [pos=0.35] (A)
            coordinate [pos=0.5]  (B)
   ;
    \xdef\slopenew{\pgfplotstableregressiona}
    \draw[color=mycolor1] (A) -| (B) node [color=mycolor1,pos=.75,anchor=east] {\small\({ \approx \pgfmathprintnumber{\slopenew} }\)};
    
\addplot [color=mycolor2,solid,line width=1.2pt,mark=o] table [x index=0, y index=1] {python/convergence_plot.txt};
\addlegendentry{\({ \textnormal{L}^1 }\) error};
\addplot [color=mycolor2,no marks,dotted,line width=1pt,forget plot] table [x=noise, y={create col/linear regression={y=l1,variance list={200,50,50,50,100,300,1000,2000}}}]
        {python/convergence_plot.txt}
            coordinate [pos=0.25] (A)
            coordinate [pos=0.4]  (B)
   ;
    \xdef\slope{\pgfplotstableregressiona}
    \draw[color=mycolor2] (B) -| (A) node [color=mycolor2,pos=.75,anchor=west] {\small\({ \approx \pgfmathprintnumber{\slope}}\)};
\end{loglogaxis}%
\end{tikzpicture}%
%
\caption{Observed convergence rates in the \({ \textnormal{L}^2 }\) and \({ \textnormal{L}^1 }\) norms for the parameter-identification problem~\eqref{eq:Lav_parid} with~\eqref{eq:PDE-example} in~\eqref{eq:PDE}.}
\label{fig:convergence-plot}
\end{wrapfigure}

With regards to convergence rates,
it is not yet clear how one can verify the conditions in \cref{thm:rates-norm}.
Nevertheless, we demonstrate numerical convergence in \cref{fig:convergence-plot},
with estimated rates in the \({ \textnormal{L}^2 }\) (that is,~\({ \X }\)) and \({ \textnormal{L}^1 }\) norms.
Note that the \({ \textnormal{L}^1 }\)~rate is approximately twice that of the \({ \textnormal{L}^2 }\)~rate.
We suspect that this behaviour is due to how
errors in the estimated positions of the jumps of the solution
are treated by the \({ \textnormal{L}^2 }\) and  \({ \textnormal{L}^1 }\) norm, respectively:
If the true solution has a jump at a position~\({ x_0 }\), but
the reconstruction places it at \({x_0 \pm \varepsilon }\),
the resulting \({ \textnormal{L}^1 }\)-error will be of order~\({ \varepsilon^1 }\),
whereas the \({ \textnormal{L}^2 }\)-error is only of order~\({ \varepsilon^{1/2} }\).
We remark also that the deviations for small noise levels~\({ \delta }\) are due to the fact that the PDE-reconstructions~\({ y_\alpha^\delta }\) (the solution of~\eqref{eq:PDE} corresponding to the reconstruction~\({ u_\alpha^\delta = \imath v_\alpha^\delta }\)) catch up with the finite resolution of the reference PDE-solution~\({ y^\dagger }\).

\section{Conclusions and further work}

In this paper we have introduced subgradient-based Lavrentiev regularisation as an alternative to Tikhonov regularisation for linear and nonlinear inverse problems with monotone forward operators.
We have established a well-posedness theory in Banach spaces subject to weak assumptions on both the forward operator and the regulariser,
and proved general convergence-rates results under variational source conditions.
In addition, we have studied examples with TV denoising in convolutional Volterra-integral equations
and nonlinear parameter-identification problems in parabolic PDEs.

These results open up for new investigations.
For instance, one may consider generalisations that include noisy forward operators.
It would also be interesting to study the case of nonreflexive Banach spaces or pseudomonotone operators,
which both seem apparently much harder.
With regards to applications, it is desirable to obtain relevant monotonicity results also for nonlinear integral operators
such as autoconvolution and Hammerstein--Volterra operators, or higher-dimensional extensions.
Finally, we would also like to find practical characterisations of the variational source conditions,
besides adaptive selection criteria for the regularisation parameter.

\appendix

\section{Proof of \cref{thm:alternative-growth-conditions}}
\label{sec:alternative-growth-conditions}

We show that conditions~\eqref{eq:spherical-growth} and~\eqref{eq:directional-growth}
imply the growth condition \Cref{assumptions}~\ref{assumptions:growth}
with \({ t_0 = 1 }\).

For convenience, note first that~\eqref{eq:directional-growth} can equivalently be stated as
\begin{equation} \label{eq:directional-growth-variant}
  \dualitypairing[\big]{\mathcal{A}(u_{\textnormal{ref}} + w) - \mathcal{A}u_{\textnormal{ref}}}{w}
  \geq \gamma(\norm{w}) \dualitypairing[\Big]{\mathcal{A}\Bigl(u_{\textnormal{ref}} + \frac{w}{\norm{w}}\Bigr) - \mathcal{A}u_{\textnormal{ref}}}{ \frac{w}{\norm{w}}}
\end{equation}
for all \({ w \in \X }\) with \({ \norm{ w } }\) sufficiently large.
Now let \({ C > 0 }\) and define
\begin{equation} \label{eq:w-affine}
w_{r,u} \coloneqq u^\dagger - u_{\textnormal{ref}} + r(u-u^\dagger),
\end{equation}
for \({ u \in U_C }\) and \({ r > 0 }\).
By \Cref{assumptions}~\ref{assumptions:regulariser} and~\ref{assumptions:coercivity} the set \({ U_C }\) is compact,
which implies that
\[
  K \coloneqq \Set[\bigg]{ \frac{w_{r,u}}{\norm{w_{r,u}}} \given u \in U_C \text{ and } r \geq \norm{u^\dagger - u_{\textnormal{ref}}} + 1},
\]
is a precompact subset of the unit sphere in~\({ \X }\).
Since \({ \mathcal{A} }\) is demicontinuous (being monotone and hemicontinuous),
the mapping
\[
  v \mapsto \dualitypairing[\big]{\mathcal{A}(u_{\textnormal{ref}}+v) - \mathcal{A}u_{\textnormal{ref}}}{v}
\]
is continuous,
and because of the strict monotonicity of \({ \mathcal{A} }\), it follows that
\[
  c_0 \coloneqq \inf_{v \in K} \dualitypairing[\big]{\mathcal{A}(u_{\textnormal{ref}}+v) - \mathcal{A}u_{\textnormal{ref}}}{v} > 0.
\]

In order to establish~\eqref{eq:coercivity-growth},
we see from~\eqref{eq:w-affine} that \({ ru + (1-r)u^\dagger = u_{\textnormal{ref}} + w_{r,u} }\),
meaning that
\begin{equation*}
\dualitypairing*{\mathcal{A}\bigl(ru + (1-r)u^\dagger \bigr)}{u-u^\dagger}
 = \frac{1}{r}\dualitypairing[\big]{\mathcal{A}(u_{\textnormal{ref}}+w_{r,u})}{w_{r,u}-(u^\dagger - u_{\textnormal{ref}})}.
\end{equation*}
We then decompose the latter expression as follows (ignoring \({ 1/r }\) for the moment):
\begin{multline*}
\dualitypairing[\big]{\mathcal{A}(u_{\textnormal{ref}}+w_{r,u})}{w_{r,u}-(u^\dagger - u_{\textnormal{ref}})}\\
= \underbrace{\dualitypairing[\big]{\mathcal{A}(u_{\textnormal{ref}}+w_{r,u})-\mathcal{A}u_{\textnormal{ref}}}{w_{r,u}}}_{\eqqcolon \textnormal{I}}
+ \underbrace{\dualitypairing[\big]{\mathcal{A}u_{\textnormal{ref}}}{w_{r,u}}}_{\eqqcolon \textnormal{II}}
- \underbrace{\dualitypairing[\big]{\mathcal{A}(u_{\textnormal{ref}}+w_{r,u})}{u^\dagger-u_{\textnormal{ref}}}}_{\eqqcolon \textnormal{III}}.
\end{multline*}
The growth condition~\eqref{eq:directional-growth-variant} and the definition of \({ c_0 }\) imply that
\begin{equation*}
\textnormal{I} \geq \gamma(\norm{w_{r,u}}) \dualitypairing[\Big]{\mathcal{A}\Bigl(u_{\textnormal{ref}}+\frac{w_{r,u}}{\norm{w_{r,u}}}\Bigr)-\mathcal{A}u_{\textnormal{ref}}}{\frac{w_{r,u}}{\norm{w_{r,u}}}} \geq c_0 \gamma(\norm{w_{r,u}})
\end{equation*}
whenever \({ r \geq \norm{u^\dagger - u_{\textnormal{ref}}} + 1 }\) (so that \({ w_{r, u} / \norm{ w_{r, u} } \in  K }\)).
Moreover,
\begin{equation*}
\textnormal{II} \geq - \norm{ \mathcal{A}u_{ \textnormal{ref}} }_{\X^*} \norm{ w_{r, u} }
\qquad \text{and} \qquad
\textnormal{III} \leq \norm{ \mathcal{A}(u_{ \textnormal{ref}} + w_{r, u}) }_{\X^*} \norm{ u^\dagger - u_{ \textnormal{ref}} },
\end{equation*}
so that, in total,
\begin{align} \label{eq:growth-middleterms}
\begin{aligned}
\dualitypairing*{\mathcal{A}\bigl(ru + (1-r)u^\dagger \bigr)}{u-u^\dagger}
&= \frac{ 1 }{ r } \left( \textnormal{I} + \textnormal{II} - \textnormal{III} \right) \\[1ex]
& \geq \frac{ \gamma(\norm{w_{r,u}}) }{ r } \left( c_0 - \frac{ \norm{ \mathcal{A}(u_{ \textnormal{ref}} + w_{r, u}) }_{\X^*} }{ \gamma(\norm{w_{r,u}}) } \, \norm{ u^\dagger - u_{ \textnormal{ref}} }  \right) \\[1ex]
&\phantom{= {}} - \norm{ \mathcal{A}u_{ \textnormal{ref}} }_{\X^*} \frac{ \norm{ w_{r, u} } }{ r }.
\end{aligned}
\end{align}
The latter term is uniformly bounded for all~\({ r > 2\norm{u^\dagger - u_{\textnormal{ref}}} }\) (and \({ u \in U_C }\)),
as seen from
\[
  \frac{\norm{w_{r,u}}}{r} = \norm[\Big]{\frac{u^\dagger - u_{\textnormal{ref}}}{r} + u-u^\dagger},
\]
which together with \({ \norm{u-u^\dagger} = 1 }\) implies that \({ \frac{1}{2} \leq \frac{\norm{w_{r,u}}}{r} \leq \frac{3}{2} }\).
Since
\begin{equation*}
\sup_{ u \in U_C } \frac{ \norm{ \mathcal{A}(u_{ \textnormal{ref}} + w_{r, u}) }_{\X^*} }{ \gamma(\norm{w_{r,u}}) } \xrightarrow[ r \to \infty ]{  } 0 
\end{equation*}
due to~\eqref{eq:spherical-growth},
it follows that the big parenthesis in~\eqref{eq:growth-middleterms}
eventually becomes positive for sufficiently large \({ r > 0 }\) independent of~\({ u }\).
Thus the entire right-hand side converges to~\({ \infty }\) as~\({ r \to \infty }\), uniformly over~\({ u \in U_C }\), in light of the superlinear growth of~\({ \gamma }\).
This proves the growth condition~\eqref{eq:coercivity-growth} in \Cref{assumptions}~\ref{assumptions:growth}.

\begin{remark}
As seen from~\eqref{eq:growth-middleterms}, the spherical growth condition~\eqref{eq:spherical-growth} is unnecessary if~\({ u_{ \textnormal{ref}} = u^\dagger }\).
\end{remark}


\begin{multicols}{2}[{\printbibheading\vspace*{-1em}}]
\begin{spacing}{1.12}
\printbibliography[heading=none]
\end{spacing}
\end{multicols}

\end{document}